\documentclass[12pt]{amsart}
\usepackage{fullpage}
\usepackage{amssymb}
\usepackage{hyperref}
\usepackage[all]{xy}
\usepackage{dsfont}
\usepackage{xspace}
\allowdisplaybreaks
\numberwithin{equation}{section}
\newtheorem{lemma}[equation]{Lemma}
\newtheorem{theorem}[equation]{Theorem}
\newtheorem{prop}[equation]{Proposition}

\theoremstyle{definition}
\newtheorem{defn}[equation]{Definition}
\newtheorem{eg}[equation]{Example}
\newtheorem{rk}[equation]{Remark}

\newtheorem{warning}[equation]{Warning}

\renewcommand{\le}{\leqslant}
\renewcommand{\ge}{\geqslant}
\newcommand{\ad}{\mathsf{ad}}
\newcommand{\Adj}{\mathrm{Adj}}
\newcommand{\Ann}{\mathsf{Ann}}
\newcommand{\As}{{\mathcal{A}\hspace{-0.4mm}\mathit{s}}}
\newcommand{\CAlg}{\mathsf{CAlg}\hspace{0.3mm}}
\newcommand{\Der}{\mathsf{Der}}
\newcommand{\Dist}{\mathsf{Dist}}
\newcommand{\End}{\mathsf{End}}
\newcommand{\ep}{\varepsilon}
\newcommand{\ev}{\mathrm{ev}}
\newcommand{\GL}{\mathsf{GL}}
\newcommand{\Gr}{\mathsf{Gr}}
\newcommand{\Grp}{\mathcal{G}\hspace{-0.2mm}\mathit{rp}}
\newcommand{\Hom}{\mathsf{Hom}}
\newcommand{\Id}{\mathsf{Id}}

\newcommand{\Lie}{\mathsf{Lie}}
\newcommand{\Lieop}{\mathcal{L}\mathit{ie}}
\newcommand{\Mod}{\mathsf{Mod}}
\newcommand{\mmod}{\mathsf{mod}}
\newcommand{\one}{\mathds{1}}
\newcommand{\op}{{\mathrm{op}}}
\newcommand{\opt}{{\op\text{-}{\scriptscriptstyle\otimes}}}
\newcommand{\Set}{{\mathcal{S}\hspace{-0.4mm}\mathit{et}}}
\newcommand{\SV}{{\mathcal{SV}\hspace{-0.4mm}\mathit{ec}}}
\newcommand{\sv}{{\mathit{s}\nu\hspace{-0.4mm}\mathit{ec}}}
\newcommand{\Tan}{\mathcal T}
\newcommand{\Tr}{\mathsf{Tr}}
\newcommand{\VV}{{\mathcal{V}\hspace{-0.4mm}\mathit{ec}}}
\newcommand{\vv}{{\nu\hspace{-0.4mm}\mathit{ec}}}
\newcommand{\bbG}{\mathbb G}
\newcommand{\bbZ}{\mathbb Z}
\newcommand{\mcA}{\mathcal A}
\newcommand{\mcAi}{{\vec{\hspace{0.2mm}\mcA}}}
\newcommand{\mcB}{\mathcal B}
\newcommand{\mcBi}{{\vec{\hspace{0.2mm}\mcB}}}
\newcommand{\mcC}{\mathcal C}
\newcommand{\mcCi}{{\vec{\hspace{0.2mm}\mcC}}}
\newcommand{\mcE}{\mathcal E}
\newcommand{\mcO}{\mathcal O}
\newcommand{\Ver}{\mathsf{Ver}}
\newcommand{\Veri}{\overrightarrow{\mathsf{V}{\mathsf{e}}\mathsf{r}}}
\newcommand{\CV}{{\Ver_4^+}}
\newcommand{\CVi}{{\Veri_4^+}}
\newcommand{\CVAlg}{\CAlg_4^+}

\newcommand{\mfm}{\mathfrak m}
\newcommand{\Sym}{\Sigma}
\newcommand{\lb}{{[\![}}
\newcommand{\rb}{{]\!]}}
\newcommand{\EGNO}{\cite{Etingof/Gelaki/Nikshych/Ostrik:2015a}\xspace}
\author{Dave Benson}
\author{Julia Pevtsova}
\title{Group schemes and their Lie algebras over a symmetric tensor category}
\begin{document}

\begin{abstract}
We investigate the theory of affine group schemes over a symmetric tensor
category, with particular attention to the tangent space at the
identity.  We show that this carries the structure of a restricted Lie
algebra, and can be viewed as the degree one distributions on the
group scheme, or as the right invariant derivations on the coordinate
ring. In the second half of the paper, we illustrate the theory in the
particular case of the symmetric tensor category $\mathsf{Ver}_4^+$ in
characteristic two.
\end{abstract}



\maketitle

\tableofcontents

\section{Introduction}

Let $\mcC$ be a rigid $k$-linear symmetric tensor category,
and write $\mcCi$ for its ind-completion. For example, if $\mcC$ is
the category $\vv(k)$ of finite dimensional vector spaces then $\mcCi$
is the category $\VV(k)$ of all vector spaces.

The theory of group schemes and the theory of supergroup schemes
are examples of a more general theory, namely group schemes defined over 
$\mcCi$. In the case of group schemes
we have $\mcC=\vv(k)$, while 
in the case of supergroup schemes we have $\mcC=\sv(k)$, the category 
of finite dimensional super vector spaces, with $\mcCi=\SV(k)$ all
super vector spaces. Our goal here is to write a version of the
theory for general $\mcC$. We begin the paper with some introductory
material on commutative algebras and group schemes over finite tensor
categories, in the first few sections. The reader can also consult the 
references~\cite{Coulembier:ca,Coulembier/Sherman:hstc,Venkatesh:2023a}.

Much of the theory mimics what happens in the case where $\mcCi=\VV(k)$,
with suitably categorical versions of the definitions. But there are
some subtleties. One of them is the definition of tangent space, which
we give in Section~6. For this purpose, we introduce a categorical
notion $\mcE_\mcC$ of the ring of dual numbers. We take the right
adjoint of the tensor product functor, and apply it to the tensor
identity to produce a dualising object $E_\mcC$, and then $\mcE_\mcC$ is
defined to be $\one\oplus E_\mcC$, with augmentation the projection
$\eta\colon\mcE_\mcC\to \one$, and kernel $E_\mcC$ squaring to zero. 
This is a commutative algebra over
$\mcC\boxtimes \mcC$, Deligne's tensor product of 
$\mcC$ with itself. 
The tangent space is the first part of the algebra $\Dist(G)$ of
distributions at the identity on a group scheme $G$ over $\mcCi$.

We give three different interpretations of the Lie algebra $\Lie(G)$ of
$G$. The first is the tangent space at the identity, the second is the
degree one part of the algebra of distributions, and the third is the
right invariant derivations of the coordinate ring
$\mcO(G)$. Effectively, we are translating Theorem~12.2 of
Waterhouse~\cite{Waterhouse:1979a}, to the context of affine group
schemes over a symmetric tensor category.

\begin{theorem}\label{th:main}
Let $G$ be an affine group scheme over $\mcCi$. Then there are
canonical isomorphisms between the following.
\begin{enumerate}
\item The tangent space $\Lie(G)$ to $G$ at the identity, 
\item The degree one distributions on $G$. 
\item The right invariant derivations from $\mcO(G)$ to
  $\one$.
\end{enumerate}
\end{theorem}

The
definition of a Lie algebra over $\mcCi$ needs some care. There are
three distinct notions which are, from weakest to strongest,
\emph{operadic Lie algebra}, \emph{Lie algebra}, and \emph{restricted
  Lie algebra}. The notion of operadic Lie algebra is easy to define
using the Lie operad $\Lieop$.
We use Etingof's definition~\cite{Etingof:2018a} of a Lie algebra,
where non-linear identities are added to the operadic definition.
These identities are not satisfied by the free operadic Lie algebra on
an object in $\Ver_p$, for example, so Etingof's notion is strictly
stronger than the operadic one.
We use Fresse's definition~\cite{Fresse:1999a} of a restricted Lie
algebra as a $\Gamma\Lieop$-algebra, namely a divided power algebra
over $\Lieop$.
Fresse's definition of restricted Lie algebra implies but is not implied by Etingof's
definition of Lie algebra, so all these notions are distinct. 
We show that $\Lie(G)$ satisfies the strongest of these---it
is a restricted Lie algebra in $\mcCi$.

The equivalence of parts (i) and (ii) of Theorem~\ref{th:main}
is proved in Proposition~\ref{pr:Gerstenhaber1=Lie}, and the
  equivalence with (iii) is proved in
  Section~\ref{se:universality}. The fact that $\Lie(G)$ is a
  restricted Lie algebra is proved in Section~\ref{se:restricted}.

In the second part of the paper, we illustrate the theory with the particular
case of a category $\CV$ in characteristic two. This category
first appeared in work on topological K-theory in the 
nineteen sixties
\cite{Araki:1967a,
Araki/Toda:1965a,
Araki/Toda:1966a,
Araki/Yosimura:1970a,
Araki/Yosimura:1971a},
later in other K-theories including 
Morava K-theory~\cite{
Kane:1988a,
Kultze/Wurgler:1987a,
Kultze/Wurgler:1988a,
Mironov:1979a,
Nassau:2002a,
Strickland:1999a,
Wurgler:1977a,
Wurgler:1986a}
and was further investigated more recently by Kaufer~\cite{Kaufer:sc2},
Venkatesh~\cite{Venkatesh:2016a} and Hu~\cite{Hu:GLVer4+,Hu:Ver4+}.
This is the second in an infinite sequence of incompressible symmetric
tensor categories in characteristic two investigated 
in~\cite{Benson/Etingof:2019a,Benson/Etingof:2022a}. The corresponding
categories in odd characteristic
appear in~\cite{Benson/Etingof/Ostrik:2023a,Coulembier:2021a}.
The second part of the paper reworks the first part
more concretely in the special case of $\CV$, so that the reader can
see exactly what the theory is saying in this particular example. This
means that some details are reproved there with elements rather than
functors. Doing these calculations helped the authors see how to
formulate Part~1 of this paper.\bigskip

\noindent
{\bf Acknowledgements.} The first named author thanks the University
of Washington, and both authors thank City, University of London, for
their hospitality during the preparation of this paper. This research
was partly supported by the NSF Awards DMS-1901854 and DMS-2200832.
The authors would also like to thank Pavel Etingof for useful feedback.

\part{\texorpdfstring{\centering General theory}
{General Theory}}

\section{Symmetric tensor categories}

For preliminaries on tensor categories, we refer to Deligne~\cite{Deligne:2002a}
and the book of Etingof et al.~\EGNO.

Let $k$ be an algebraically closed field, and let $\mcCi$ be an 
abelian category with set indexed 
direct sums. Let $\mcC$ be
the full subcategory of compact objects in $\mcCi$,
namely the objects $X$ such that the natural map
\[ \bigoplus_{i\in I}\Hom(X,Y_i) \to \Hom(X,\bigoplus_{i\in I} Y_i) \]
is an isomorphism, whenever the $Y_i$ are objects in $\mcCi$
indexed by a set $I$. We suppose that $\mcC$ generates $\mcCi$.

We suppose that $\mcC$ is a finite symmetric tensor category.
Thus $\mcC$ is an abelian category that comes with
a tensor product, which is a functor of two variables, exact in 
each variable, sending
$V$ and $V'$ to $V \otimes V'$, together with a symmetric
braiding $s\colon V \otimes V' \to V' \otimes V$ and a simple unit
object $\one$, together with maps $l_V\colon \one\otimes V \to V$
and $r_V\colon V \otimes\one \to V$ such that
the usual commutativity, associativity and unit diagrams hold
up to coherent natural isomorphism. The tensor product and symmetric
braiding extend naturally to $\mcCi$. The following diagram commutes
for all objects $V$ in $\mcCi$, see Proposition XIII.1.2 of Kassel~\cite{Kassel:1995a}.
\begin{equation}\label{eq:one-tensor}
\vcenter{\xymatrix{\one \otimes V \ar[rr]^s\ar[dr]_{l_V} 
&& V \otimes \one\ar[dl]^{r_V} \\ &V}}
\end{equation}

We recall that by definition, a tensor category is rigid, meaning that
all objects are rigid.
An object $V$ is said to be \emph{rigid} if there exist a (necessarily
unique up to isomorphism) 
left dual, denoted $V^*$, and a right dual, denoted ${}^*V$,
together with natural maps
$\ep_V\colon V^*\otimes V \to \one$, $\ep'_V\colon V \otimes {}^*V \to
\one$, $\eta_V\colon\one \to V \otimes V^*$ and
$\eta'_V\colon \one \to {}^*V \otimes V$ such that the composites
\begin{align*}
V \xrightarrow{\eta_V\otimes 1} V \otimes V^* \otimes V 
\xrightarrow{1 \otimes \ep_V} V &&
V^*\xrightarrow{1 \otimes\eta_V} V^* \otimes V \otimes V^*
\xrightarrow{\ep_V\otimes 1} V^* \\
V \xrightarrow{1\otimes\eta'_V} V \otimes {}^*V \otimes V 
\xrightarrow{\ep'_V\otimes 1} V &&
{}^*V\xrightarrow{\eta'_V\otimes 1} {}^*V \otimes V \otimes {}^*V
\xrightarrow{1\otimes\ep'_V} {}^*V
\end{align*}
are identity maps.
We have adjunctions
\begin{align*}
  \Hom(U \otimes V,W)&\cong \Hom(U,W\otimes V^*), \\
  \Hom(V \otimes U,W)&\cong \Hom(U,{}^*V\otimes W).
\end{align*}

In the presence of the symmetric braiding, the left and right dual are
equal, and we just write $V^*$.

Finally, we note that $\mcC$ is equivalent as an
abelian category to the category $\mmod(B)$ of finitely generated
modules for a finite dimensional
algebra $B$. The category $\mcCi$ can then be viewed as the category
of ind-objects in $\mcC$, so it is equivalent to the category
$\Mod(B)$ of all $B$-modules.

\section{Algebras and commutative algebras}\label{se:CA}

\begin{defn}
An \emph{algebra} over $\mcCi$ is an object $A$ in $\mcCi$ together 
with morphisms $\iota \colon \one \to A$ and 
$\mu\colon A \otimes A\to A$ satisfying the
usual associativity and unit conditions. The tensor product of
two algebras $A$ and $A'$ over $\mcCi$ is given the structure of
an algebra using the symmetric braiding:
\begin{equation}\label{eq:muAtensorA'} 
\xymatrix{A \otimes A' \otimes A \otimes A' \ar[r]^{1 \otimes s \otimes 1} 
\ar[dr]_{\mu_{A\otimes A'}} &
A \otimes A \otimes A' \otimes A' \ar[d]^{\mu_A\otimes\mu_{A'}} \\
& A \otimes A'.} 
\end{equation}
\end{defn}

\begin{rk}\label{rk:iota}
It follows from~\eqref{eq:one-tensor} and functoriality of $s$ that
the following diagrams commute.
\[ \xymatrix@=6mm{\one\otimes A\ar[r]^(.55){l_A}_(.55)\simeq\ar[d]_{\iota\otimes\Id}&
A&
A \otimes \one \ar[l]_(.55){r_A}^(.55)\simeq\ar[d]^{\Id\otimes\iota}\\
A\otimes A \ar[rr]^s && A\otimes A}\qquad
 \xymatrix@=6mm{
A\otimes A \ar[d]_{\eta\otimes\Id} \ar[rr]^s && 
A\otimes A \ar[d]^{\Id\otimes\eta} \\
\one\otimes A\ar[r]^(.55){l_A}_(.55)\simeq&A&
A \otimes \one \ar[l]_(.55){r_A}^(.55)\simeq
} \]
\end{rk}

\begin{defn}
An algebra $A$ over $\mcCi$ is said to be 
\emph{commutative} if the following diagram commutes:
\[ \xymatrix{A \otimes A \ar[dr]^\mu\ar[d]_s  \\ A \otimes A
    \ar[r]_\mu&A.} \]
\end{defn}

\begin{lemma}
An algebra $A$ over $\mcCi$ is commutative if and only if
the multiplication map $\mu\colon A \otimes A \to A$ is a
morphism of algebras over $\mcCi$.
\end{lemma}
\begin{proof}
Consider the following diagram.
\[ \xymatrix{
&&A \otimes A \otimes A \otimes A \ar[d]^{\mu_A\otimes\mu_A}\\
&A \otimes A \otimes A \otimes A \ar[ur]^{1 \otimes s \otimes 1\quad} 
\ar[r]_(.6){\mu_{A\otimes A}} \ar[d]^{\mu_A\otimes \mu_A}& A \otimes A \ar[d]^{\mu_A} \\
\one \otimes A \otimes A \otimes \one \ar[r]_(.6){\cong}
\ar[ur]^{\iota\otimes 1 \otimes 1 \otimes \iota\quad } 
& A \otimes A\ar[r]^(.55){\mu_A}&A} \]
The two small triangles commute by definition. The square commutes if and
only if $\mu_A$ is a morphism of algebras over $\mcCi$. The large triangle
commutes if and only if $A$ is a commutative algebra over $\mcCi$.
\end{proof}

\begin{defn}
Commutative algebras over $\mcCi$ form a category which is
denoted $\CAlg_\mcCi$. There is a natural functor from
$\CAlg=\CAlg_{\VV(k)}$ to $\CAlg_\mcCi$ coming from the inclusion 
$\VV\to\mcCi$ that sends $k$ to $\one$. It is a full embedding with a left adjoint and
a right adjoint.

A $\bbZ$-graded commutative algebra over $\mcCi$ consists
of objects $A_n$, $n\in\bbZ$, together with morphisms
$\one \to A_0$ and $\mu\colon A_i \otimes A_j \to A_{i+j}$
making $\bigoplus_n A_n$ into a commutative algebra over $\mcCi$.
The category $\Gr\CAlg_\mcCi$ has as objects the $\bbZ$-graded
commutative algebras over $\mcCi$ and the 
morphisms are morphisms of algebras 
over $\mcCi$ that preserve the grading.
\end{defn}

\begin{eg}
For any object $X$ in a braided tensor category, there is
an action of the braid group $B_n$ on the tensor $n$th
power $X^{\otimes n}$. In the case of our symmetric tensor
category $\mcCi$, this action factors through the map from $B_n$
to the symmetric group $\Sym_n$. 
The $n$th \emph{symmetric power}
of $X$, denoted $S^n(X)$, is the largest quotient of $X^{\otimes n}$ on which $\Sym_n$
acts trivially. There are natural maps $S^m(X) \otimes S^n(X) \to S^{m+n}(X)$
which satisfy the commutative and associative laws. 
The \emph{symmetric algebra} $S(X)$ on an object $X$ in $\mcCi$
is defined to be $\bigoplus_{n=0}^\infty S^n(X)$, with these natural 
maps giving the multiplication. The algebra $S(X)$ is an object in
$\CAlg_\mcCi$, which may be regarded as the \emph{free} 
commutative algebra on $X$.
\end{eg}

\begin{lemma}
If $A$ and $A'$ are commutative algebras over $\mcCi$, then so is $A\otimes A'$.
\end{lemma}
\begin{proof}
Consider the following diagram.
\begin{equation}\label{eq:AtensorA'}
\vcenter{\xymatrix{(A\otimes A') \otimes (A\otimes A') 
\ar[r]^(.55){1\otimes s \otimes 1} \ar@/_/[d]^s
\ar@/^11ex/[drr]^\mu &
A\otimes A \otimes A' \otimes A' \ar[dr]^(.55){\mu\otimes\mu} 
\ar[d]^{s\otimes s} \\
(A\otimes A') \otimes (A\otimes A') \ar[r]_(.55){1\otimes s \otimes 1}
\ar@/_7ex/[rr]_\mu &
A\otimes A \otimes A'\otimes A'\ar[r]_(.6){\mu\otimes\mu} &
A \otimes A'} }
\end{equation}
The commutativity of the outer curved triangle follows from the commutativity
of the smaller square and three triangles constituting it.
\end{proof}

\section{Hopf algebras}\label{se:Hopf}

\begin{defn}
A \emph{bialgebra} over $\mcCi$ is an algebra $A$ in $\mcCi$
together with a counit $\eta\colon A \to \one$ and a 
coassociative comultiplication $\Delta\colon A \to A\otimes A$,
which are morphisms of algebras over $\mcCi$. Thus the
following diagrams commute:
\[ \vcenter{\xymatrix{A \ar[r]^\Delta\ar[d]_\Delta & A \otimes A \ar[d]^{\eta\otimes 1} \\
    A \otimes A \ar[r]_{1\otimes\eta} & A}} \qquad
\vcenter{\xymatrix{A \otimes A \ar[r]^(.38){\Delta\otimes\Delta}
\ar[dd]_\mu & A \otimes A
    \otimes A \otimes A \ar[dr]^{1 \otimes s \otimes 1} \\ && A \otimes
    A \otimes A \otimes A\ar[d]^{\mu\otimes\mu} \\
A \ar[rr]^{\Delta} && A \otimes A}} \]
where $\mu$ is the multiplication map. The second of these diagrams
states that the comultiplication is a morphism of algebras in $\mcCi$, or
equivalently that the multiplication is a morphism of coalgebras in $\mcCi$.
\end{defn}

\begin{defn}
A \emph{Hopf algebra} over $\mcCi$ is a bialgebra $A$ over $\mcCi$ 
together with an \emph{antipode} $\tau\colon A \to A$ such that
the following diagram commutes:
\[ \xymatrix{A \otimes A\ar[dd]^{1\otimes\tau} & 
A \ar[d]^\eta \ar[l]_(.4)\Delta\ar[r]^(.4)\Delta & A \otimes A \ar[dd]^{\tau
      \otimes 1} \\ & \one \ar[d]^\iota & \\
A \otimes A\ar[r]^(.55)\mu & A & A \otimes A.\ar[l]_(.55)\mu} \]
A Hopf algebra $A$ is \emph{cocommutative} if the following 
diagram commutes:
\[ \xymatrix{A \ar[r]^\Delta \ar[dr]_\Delta & A \otimes A \ar[d]^s \\ 
& A \otimes A.} \]
\end{defn}

\begin{eg}\label{eg:Hopf-superalgebra}
A \emph{Hopf superalgebra} is a 
Hopf algebra over the symmetric tensor
category $\SV$ of super vector spaces. A Hopf superalgebra cannot be
considered as a Hopf algebra over $\VV$. For example, an exterior
algebra is a Hopf superalgebra, but it is not a Hopf algebra because
the comultiplication is not a morphism of algebras.
\end{eg}

\section{The canonical algebra}

Deligne's tensor product $\boxtimes$ is defined for locally finite 
abelian categories, see \EGNO, \S1.11. So by $\mcAi\boxtimes\mcBi$ we mean 
ind-objects in $\mcA\boxtimes\mcB$. 
Let $\mcC^\opt$ denote the tensor category obtained by 
reversing the order of the tensor product (and \emph{not} the opposite category). 
The tensor product is an exact functor $\mcCi\boxtimes\mcCi^\opt\to\mcCi$
of $\mcCi\boxtimes\mcCi^\opt$-module categories, and restricts to an 
exact functor on compact objects $\mcC\boxtimes\mcC^\opt\to\mcC$. 
By the representability  
statement in \EGNO, Theorem~1.8.10 and Corollary~1.8.11,  the tensor
product on $\mcC$ has a right adjoint functor
\[ E\colon\mcC\to\mcC\boxtimes\mcC^\opt. \]
Thus we have 
\begin{equation}\label{eq:E}
  \Hom_{\mcC\boxtimes\mcC^\opt}(V \boxtimes W,E(X)) \cong 
  \Hom_\mcC(V \otimes W,X). 
\end{equation}

\begin{defn}\label{def:EC}
The image $E_\mcC=E(\one)$ of $\one$ in $\mcC\boxtimes\mcC^\opt$
 constructed this way is called the 
\emph{canonical algebra} of $\mcCi$, see \EGNO Definition~7.9.12 and 
\S\S7.18--20. It is a commutative Hopf algebra in $\mcC\boxtimes
\mcC^\opt$, see \S7 of Deligne~\cite{Deligne:1990a}, and also
\S11 of~\cite{EntovaAizenbud/Hinich/Serganova:preprint}.
\end{defn}

We define $\Hom_\mcC(V,E_\mcC)$ is as an internal Hom, namely a 
right adjoint defined by 
\[ \Hom_{\mcC^\opt}(W,\Hom_\mcC(V,E_\mcC)) \cong 
\Hom_{\mcC\boxtimes\mcC^\opt}(V \boxtimes W,E_\mcC) \]
for $W\in\mcC^\opt$, again using representability. 

Using the adjunction~\eqref{eq:E}, the right hand side is isomorphic to 
\[ \Hom_{\mcC^\opt}(W\otimes V,\one) \cong \Hom_{\mcC^\opt}(W, {}^*V) \]
(the right dual ${}^*V$ is taken in $\mcC$, and is the same as the left dual
in $\mcC^\opt$).
By Yoneda's lemma, it follows that 
\[ {}^*V\cong\Hom_\mcC(V,E_\mcC). \]

Since we are assuming that the braiding $s$ is symmetric,
we may identify 
$\mcC^\opt$ with $\mcC$ and ${}^*V$ with $V^*$, and we consider
$E_\mcC$ as an object in $\mcC\boxtimes\mcC$. Thus we have the following.

\begin{prop}\label{pr:EC}
For objects $V$ in the rigid symmetric tensor category $\mcC$,
there is a natural isomorphism 
\[ V^*\cong\Hom_\mcC(V,E_\mcC). \]
\end{prop}

\begin{eg}
Suppose that $\mcC$ is the category of finitely generated modules for
a finite group scheme 
$G$. Then the object $E_\mcC$ is the coordinate ring $k[G]$, which is
the dual Hopf algebra of the group algebra $kG$. The isomorphism of
Proposition~\ref{pr:EC} becomes 
\[ \Hom_{kG}(V,k[G]) \cong \Hom_{kG}(V,\Hom_k(kG,k))\cong
  \Hom_k(kG\otimes_{kG}V,k) \cong \Hom_k(V,k) \cong V^*. \]
\end{eg}

\begin{rk}
We shall use the dualising object $E_\mcC$ of Proposition~\ref{pr:EC} to construct the
appropriate analogue of the ring of dual numbers when we discuss
tangent spaces in Section~\ref{se:tangent}.
\end{rk}

\section{Affine schemes and tangent spaces}\label{se:tangent}

\begin{defn}
An \emph{affine scheme} $S$ over $\mcCi$ is a representable functor
\[ S \colon \CAlg_\mcCi \to \Set. \]
In other words, there exists a representing object $\mcO(S)$ in $\CAlg_\mcCi$,
called the \emph{coordinate ring} of $S$, satisfying 
\[ S(A)=\Hom_{\CAlg_\mcCi}(\mcO(S),A). \]
A morphism of affine schemes is a natural transformation between the functors.
\end{defn}

By Yoneda's lemma, the morphisms of affine schemes from $S$ to $S'$
are in one to one correspondence with the morphisms in $\CAlg_\mcCi$
from $\mcO(S')$ to $\mcO(S)$.

\begin{defn}\label{def:dual-numbers}
We define the \emph{ring of dual numbers} over $\mcCi$, denoted
$\mcE_\mcC$, to be the commutative algebra over $\mcC \boxtimes\mcC$
given by $\one\oplus E_\mcC$, where 
$E_\mcC$ is as in Proposition~\ref{pr:EC}, and the multiplication is
given by
\begin{align*}
(\one\oplus E_\mcC) \otimes (\one\oplus E_\mcC) 
&\cong (\one \otimes\one) \oplus (\one \otimes E_\mcC)
\oplus (E_\mcC\otimes \one) \oplus (E_\mcC \otimes E_\mcC) \\
&\xrightarrow{l_\one + l_{E_\mcC} + r_{E_\mcC} + 0}
\one \oplus E_\mcC. 
\end{align*}
Then $\mcE_\mcC$ comes with an augmentation
$\eta\colon \mcE_\mcC \to \one$, and the kernel $E_\mcC$ squares to zero. 
\end{defn}

If $S$ is an affine scheme over $\mcCi$, and $x\in
S(\one)=\Hom_{\CAlg_\mcCi}(\mcO(S),\one)$ 
with kernel $\mfm_x$, then the \emph{tangent space}
to $S$ at $x$ is the object in $\mcCi$ defined by
\begin{equation}\label{eq:tangent} 
\Tan_{S,x} = (\mfm_x/\mfm_x^2)^* 
\cong \Hom_\mcCi(\mfm_x/\mfm_x^2,E_\mcC)
\cong \Hom_{\CAlg_\mcCi}(\mcO(S),\mcE_\mcC)_x=S(\mcE_\mcC)_x.
\end{equation}
Here, $\Hom_{\CAlg_\mcCi}(\mcO(S),\mcE_\mcC)_x$, or $S(\mcE_\mcC)_x$, denotes the
fibre over $x$ of the map $\eta_*\colon S(\mcE_\mcC)\to S(\one)$.

The tangent space is the first part of the 
\emph{space of distributions} at $x$. We define
$\Dist_n(S,x)$ to be $(\mcO(S)/\mfm_x^n)^*$. 
Then $\Dist_n(S,x)\subseteq\Dist_{n+1}(S,x)$, and
we define
\[ \Dist(S,x)=\bigcup_{n=1}^\infty \Dist_n(S,x). \]
Each $\Dist_n(S,x)$ is a $\mcO(S)$-module, as is $\Dist(S,x)$,
as follows. The multiplication in $\mcO(S)$ induces a 
map $\mcO(S)\otimes \mcO(S)/\mfm_x^n \to \mcO(S)/\mfm_x^n$,
and hence by adjunction we obtain maps
\[ \mcO(S)\otimes \Dist_n(S,x)\to \Dist_n(S,x), \]
compatible with the inclusions.

We have $\Dist_0(S,x)\cong \one$, and for any $n$, 
$\Dist_n(S,x)\cong \one \oplus \Dist^+_n(S,x)$, where
$\Dist^+_n(S,x)$ is the subspace defined by
vanishing at the identity element of $\mcO(S)$.
Taking the union, we have $\Dist(S,x)=\one\oplus \Dist^+(S,x)$.
Thus $\Dist^+_1(S,x)$ is the tangent space $\Tan_{S,x}$.

\section{Group schemes}\label{se:group-schemes}

Our approach immitates the one presented in Chapter~1 of 
Jantzen~\cite{Jantzen:2003a}, but without using elements.

\begin{defn}
An \emph{affine group scheme} $G$ over $\mcCi$ is a functor $\CAlg_\mcCi\to\Grp$
such that composing with the forgetful functor $\Grp \to \Set$
gives a representable functor.
\end{defn}

By Yoneda's lemma, the natural transformations
of functors to $\Set$ corresponding to multiplication, identity and inverse maps 
in $G$ are represented by homomorphisms in $\CAlg_\mcCi$:
\[ \Delta\colon \mcO(G) \to \mcO(G)\otimes \mcO(G), \qquad
\eta\colon \mcO(G)\to \one,\qquad
 \tau\colon \mcO(G)\to \mcO(G) \]
which, together with the unit map $\iota \colon \one \to \mcO(G)$, 
satisfy the commutative diagrams defining a commutative 
Hopf algebra in $\mcCi$, as described in Section~\ref{se:Hopf}. 
As usual, this sets up a contravariant
equivalence of categories between affine group schemes in 
$\mcCi$ and commutative Hopf algebras in $\mcCi$.

If $G$ is an affine group scheme over $\mcCi$ 
we define
$\Lie(G)$ to be the tangent space at the identity,
\[ \Lie(G)=\Tan_{G,e} =(\mfm/\mfm^2)^*\cong
  \Hom_{\CAlg_\mcCi}(\mcO(G),\mcE_\mcC)_e. \] 
Here, the identity $e\in G(\one)$ is really the map
$\eta\colon\mcO(G)\to\one$, and $\mfm=\mfm_e$ is the kernel of $\eta$.
Our goal is to define what is meant by a Lie algebra in $\mcCi$
in such a way that $\Lie(G)$ has a natural structure as a Lie algebra.

This is the first part of the 
algebra of distributions, which is defined as follows. 
We define
\[ \Dist_n(G)=\Dist_n(G,e) =(\mcO(G)/\mfm^n)^* \] 
to be the spaces of distributions at the identity.
Thus $\Dist_n(G)\subseteq \Dist_{n+1}(G)$, and 
\[ \Dist(G) = \bigcup_{n=1}^\infty \Dist_n(G). \]

Using $\iota\colon\one\to\mcO(G)$ to write $\mcO(G)$ as $\one \oplus\mfm$, 
we have
\[ \mcO(G) \otimes \mcO(G) = \one \oplus (\one\otimes \mfm) 
\oplus (\mfm \otimes \one) \oplus (\mfm\otimes \mfm). \]
The map $\Delta\colon \mcO(G)\to \mcO(G) \otimes \mcO(G)$, being an algebra
homomorphism, takes
$\mfm^n$ into the $n$th power of the ideal $\mfm_{(e,e)}$ of
$\mcO(G\times G)\cong\mcO(G)\otimes \mcO(G)$, which is
\[ (\mfm\otimes \mcO(G) + \mcO(G)\otimes \mfm)^n
= \sum_{i+j=n} \mfm^i \otimes \mfm^j
\subseteq \mcO(G)\otimes \mcO(G). \]
Taking duals, we have maps
\[ \mu\colon\Dist_i(G) \otimes \Dist_j(G) \to \Dist_{i+j}(G), \]
compatible with the inclusions.
Thus $\Dist(G)$ becomes an algebra in $\mcCi$,
called the \emph{algebra of distributions}, or the
\emph{hyperalgebra} of $G$.

\begin{lemma}\label{le:Delta} 
Under $(1-s)\circ\Delta$, the image of $\mcO(G)$ lies in
$\mfm\otimes\mfm$, and for $n\ge 1$, the image of $\mfm^n$  
lies in $\sum_{i=1}^n \mfm^i \otimes \mfm^{n+1-i}$.
\end{lemma}
\begin{proof}
Tensoring the complex $\mcO(G)\xrightarrow{\eta} \one$ with itself, and inserting the
homology, we obtain an exact sequence
\[ 0 \to \mfm\otimes\mfm \to \mcO(G)\otimes \mcO(G)
  \xrightarrow{1\otimes\eta+\eta\otimes 1} (\mcO(G)\otimes \one) \oplus
  (\one \otimes \mcO(G)) \xrightarrow{(\eta,-\eta)} \one\otimes \one \to 0. \]
We also have commutative diagrams 
\[
\xymatrix@C=1.6cm{\mcO(G)\ar[d]_\Delta \ar[dr]^{(\Id,\Id)}\\
\mcO(G)\otimes \mcO(G) \ar[r]^(.4){(\Id\otimes\eta,\eta\otimes\Id)}& 
(\mcO(G)\otimes \one)
\oplus (\one\otimes \mcO(G))} 
\]
and
\[
\xymatrix@C=2cm{\mcO(G)\ar[d]_{(\Id-\iota\eta)\otimes\iota+
\iota\otimes(\Id-\iota\eta)+(\iota\otimes\iota)\eta} \ar[dr]^{(\Id,\Id)}\\
\mcO(G)\otimes \mcO(G) \ar[r]^(.4){(\Id\otimes\eta,\eta\otimes\Id)}& 
(\mcO(G)\otimes \one)
\oplus (\one\otimes \mcO(G))} 
\]
and so $\Delta - ((\Id-\iota\eta)\otimes \iota+
\iota\otimes(\Id-\iota\eta)+(\iota\otimes\iota)\eta)$ maps $\mcO(G)$ into
$\mfm\otimes\mfm$. Since 
\[(1-s)((\Id-\iota\eta)\otimes\iota+\iota\otimes(\Id-\iota\eta)+(\iota\otimes\iota)\eta)=0 \] 
(see Remark~\ref{rk:iota}), it follows that $(1-s)\Delta\colon \mcO(G) \to \mfm \otimes \mfm$.

Next, we observe that the statement that $\Delta$ is a map of algebras
is expressed by the equation
\[ \Delta\mu=(\mu\otimes \mu)(1\otimes s\otimes
  1)(\Delta\otimes\Delta) \colon \mcO(G)\otimes \mcO(G) \to \mcO(G)\otimes
  \mcO(G). \]
Since $\mcO(G)$ is a commutative algebra in $\mcCi$, we also have
\[ s(\mu\otimes\mu)(1\otimes s\otimes 1) = (\mu\otimes\mu)(1\otimes s
  \otimes 1)(s\otimes s) \colon \mcO(G)^{\otimes 4}
 \to \mcO(G)\otimes \mcO(G). \]
Thus we have
\begin{align*}
(1-s)\Delta\mu&=(1-s)(\mu\otimes\mu)
(1\otimes s\otimes 1)(\Delta\otimes \Delta) \\
&=(\mu\otimes \mu)(1\otimes s\otimes 1)(\Delta\otimes\Delta)-
(\mu\otimes\mu)(1\otimes s \otimes 1)(s\otimes s)(\Delta\otimes\Delta)  \\
&=(\mu\otimes \mu)(1\otimes s\otimes 1)
(1^{\otimes 4}-s\otimes s)(\Delta\otimes\Delta) \\
&=(\mu\otimes \mu)(1\otimes s \otimes 1)
((1-s)\Delta\otimes \Delta + s\Delta\otimes(1-s)\Delta).
\end{align*}
Now since $\Delta$ sends $\mfm$ into $\mcO(G)\otimes \mfm+\mfm\otimes
\mcO(G)$, and by the last paragraph $(1-s)\Delta$ sends $\mfm$ into $\mfm\otimes\mfm$,
it follows that the term 
\[ ((1-s)\Delta\otimes \Delta + s\Delta\otimes(1-s)\Delta) \]
sends $\mfm\otimes\mfm$ into 
\[ \mcO(G)\otimes\mfm^{\otimes 3}+
\mfm\otimes \mcO(G)\otimes\mfm^{\otimes 2}+
\mfm^{\otimes 2}\otimes \mcO(G)\otimes\mfm+
\mfm^{\otimes 3}\otimes \mcO(G)\subseteq \mcO(G)^{\otimes 4}. \]
Applying $(\mu\otimes\mu)(1\otimes s\otimes 1)$, we see that
 $(1-s)\Delta\mu$ sends $\mfm\otimes \mfm$ into
 $\mfm^2\otimes\mfm+\mfm\otimes\mfm^2$.
It follows that $(1-s)\Delta$ sends $\mfm^2$ into
 $\mfm^2\otimes\mfm+\mfm\otimes\mfm^2$.

Applying $(1-s)\Delta\mu$ to $\mfm\otimes\mfm^{n-1}$, the same
argument inductively shows that $(1-s)\Delta$ sends $\mfm^n$ into
$\sum_{i=1}^n\mfm^i\otimes\mfm^{n+1-i}$. This completes the proof of
the lemma.
\end{proof}

\begin{prop}\label{pr:GrkG}
On the associated graded with respect to $\mfm$, 
\[ \Gr \mcO(G)=\mcO(G)/\mfm\oplus\bigoplus_{i\ge 1}\mfm^i/\mfm^{i+1}, \]
the comultiplication induced by $\Delta$ is cocommutative.
\end{prop}
\begin{proof}
This follows directly from Lemma~\ref{le:Delta}.
\end{proof}

\begin{prop}\label{pr:bracket}
The image of the commutator map 
\[ \beta =\mu\circ(1- s) \colon \Dist_i(G) \otimes \Dist_j(G) \to \Dist_{i+j}(G) \]
lies in $\Dist_{i+j-1}(G)$.

The algebra of distributions $\Dist(G)$ 
is filtered by the subspaces $\Dist_n(G)$,
and the associated graded algebra
\[ \Gr\Dist(G)= \bigoplus_{n=1}^\infty \Dist_n(G)/\Dist_{n-1}(G) \]
is commutative, i.e., an object in $\Gr\CAlg_\mcCi$.
\end{prop}
\begin{proof}
The first statement follows from 
Lemma~\ref{le:Delta},
and the second follows from the first.
\end{proof}

In particular, we obtain a bracket operation
\[ \beta = \mu\circ(1 -  s) \colon \Lie(G) \otimes \Lie(G) \to \Lie(G). \]

\section{\texorpdfstring{The Lie algebra $\Lie(G)$}
{The Lie algebra Lie(G)}}\label{se:Lie}

\begin{prop}\label{pr:Dist-identities}
The maps 
\begin{align*}
\mu&\colon \Dist(G)\otimes \Dist(G)\to \Dist(G),\\
\beta&\colon \Dist(G)\otimes \Dist(G)\to \Dist(G)
\end{align*}  
satisfy the following identities:
\begin{enumerate}
\item Poisson Identity: $\beta\circ(1\otimes\mu) = 
\mu\circ(\beta\otimes 1+ (1\otimes \beta)\circ(s\otimes 1))$,
\item Anticommutativity: $\beta\circ(1+ s)=0$,
\item The Jacobi Identity: $\beta\circ (\beta\otimes 1)\circ 
(1 + (s\otimes 1)\circ (1\otimes s) + (1\otimes s)\circ (s\otimes 1)) = 0$.
\end{enumerate}
\end{prop}
\begin{proof}
Let $A=\Dist(G)$.
The following diagram commutes:
\[ \xymatrix{A \otimes A\otimes A \ar[rr]^(.55){1\otimes\mu}
\ar@<1ex>[d]^{(1\otimes s)\circ(s\otimes 1)}
&& A \otimes A \ar@<1ex>[d]^s \\
A \otimes A \otimes A \ar[rr]_(.55){\mu\otimes 1} 
\ar@<1ex>[u]^{(s\otimes 1)\circ(1\otimes s)}&&
A \otimes A\ar@<1ex>[u]^s} \]

(i) Using the associativity of $\mu$ and the above diagram, we have
\begin{align*} 
\mu\circ(\beta\otimes 1 + (1\otimes \beta)\circ(s\otimes 1)) 
&=\mu\circ((\mu\otimes 1)\circ((1-s)\otimes 1) \\
&\qquad+ 
(1\otimes \mu) \circ(1\otimes(1-s))\circ(s\otimes 1)) \\
&=\mu\circ((1\otimes \mu) - (\mu\otimes 1)\circ(1\otimes s)\circ(s\otimes 1)) \\
&=\mu\circ((1\otimes \mu)-s\circ (1\otimes\mu)) \\
&=\mu\circ(1-s)\circ(1\otimes\mu) \\
&=\mu\circ(1\otimes\beta).
\end{align*}
This proves the Poisson Identity.

(ii) We have $\beta\circ (1+s) = \mu\circ(1-s)\circ(1+s)=\mu\circ(1-s^2)=0$.
This proves the Anticommutativity.

(iii) 
Using the associativity of $\mu$ and the above diagram, we have
\begin{align*} 
\beta\circ (\beta\otimes 1)
&= \mu\circ(1-s)\circ(\mu\otimes 1)\circ(1-s\otimes 1) \\
&=\mu\circ(1\otimes\mu)\circ(1 - (s\otimes 1)\circ(1\otimes s))\circ(1-s\otimes 1) \\
&=\mu\circ(1\otimes\mu)\circ((1-(s\otimes 1)\circ(1\otimes s))+(s\otimes 1)(1-(1\otimes s)\circ(s\otimes 1))).
\end{align*}
Since
\begin{align*} 
(1-(s\otimes 1)\circ(1\otimes s))\circ
(1 + (s\otimes 1)\circ (1\otimes s) + (1\otimes s)\circ (s\otimes 1)) &=0 \\
(1-(1\otimes s)\circ(s\otimes 1))\circ
(1 + (s\otimes 1)\circ (1\otimes s) + (1\otimes s)\circ (s\otimes 1)) &=0,
\end{align*}
the Jacobi Identity now follows.
\end{proof}

\begin{theorem}\label{th:GrDist}
The associated graded algebra $\Gr\Dist(G)$ has 
a commutative multiplication 
\[ \mu\colon \Gr\Dist_i(G) \otimes \Gr\Dist_j(G) \to \Gr\Dist_{i+j}(G) \] 
and a
bracket operation 
\[ \beta \colon \Gr\Dist_i(G) \otimes \Gr\Dist_j(G) \to \Gr\Dist_{i+j-1}(G) \]
satisfying the properties {\rm (i)--(iii)} of Proposition~\ref{pr:Dist-identities}.
\end{theorem}
\begin{proof}
The properties (i)--(iii) follow from the corresponding identities in
Proposition~\ref{pr:Dist-identities} by passing to the
associated graded. The commutativity of the associated graded was proved
in Proposition~\ref{pr:bracket}.
\end{proof}

\begin{defn}\label{def:OpLie}
An \emph{operadic Lie algebra} over $\mcCi$ is an object $L$ in $\mcCi$
together with a bracket operation $\beta\colon L \otimes L \to L$ 
satisfying the following axioms:
\begin{enumerate}
\item Anticommutativity: $\beta\circ(1+ s)=0$,
\item The Jacobi Identity: $\beta\circ (\beta\otimes 1)\circ 
(1 + (s\otimes 1)\circ (1\otimes s) + (1\otimes s)\circ (s\otimes 1)) = 0$.
\end{enumerate}
\end{defn}

\begin{rk}
An equivalent definition of an operadic Lie algebra over $\mcCi$ is
an algebra over the Lie operad $\Lieop$, as defined by Ginzburg and
Kapranov~\cite{Ginzburg/Kapranov:1994a}. 
\end{rk}

\begin{defn}\label{def:Lie}
A \emph{Lie algebra} is an operadic Lie algebra satisfying some more
identities in degrees divisible by $p$. 
For example, for $\mcCi=\VV$
in characteristic two, there is one extra identity in degree two, namely
$[x,x]=0$ for all
$x\in L$. 
For super Lie algebras in characteristic three, there is also just one extra
identity, namely $[[x,x],x]=0$ for odd degree elements $x\in L$. 
It is hard to
write down the identities explicitly, and they are not operadic.
For a full discussion of this, see Section~4 of
Etingof~\cite{Etingof:2018a}, where the identities are defined
abstractly, and some of these extra identities are illustrated in the
case of Lie algebras over $\Ver_p$. 

We shall return to this topic in Section~\ref{se:restricted}, where we
shall discuss the definition of restricted Lie algebras, and show that
a restricted Lie algebra in $\mcCi$ automatically satisfies the extra identities.
\end{defn}

\begin{prop}\label{pr:Gerstenhaber1=Lie}
If $G$ is an affine group scheme over $\mcCi$ 
then the tangent space 
\[ \Lie(G)=\Tan_1(G)=\Dist^+_1(G) \]
with the bracket operation
$\beta\colon \Lie(G) \otimes \Lie(G) \to \Lie(G)$
a Lie algebra over $\mcCi$.
\end{prop}
\begin{proof}
It is immediate from properties (ii) and (iii) 
in Theorem~\ref{th:GrDist} that $\Lie(G)$ satisfies the conditions of
Definition~\ref{def:OpLie}. For the extra identity of degree $p$ in
Definition~\ref{def:Lie}, we refer to Proposition~4.7
of~\cite{Etingof:2018a}. 
\end{proof}

We shall return to this matter in Section~\ref{se:restricted}, where
we show that in fact $\Lie(G)$ carries a structure of restricted Lie
algebra in the sense of Fresse~\cite{Fresse:1999a}, and that this
implies the Proposition~\ref{pr:Gerstenhaber1=Lie} above.

\section{\texorpdfstring{Action of $\Dist(G)$ and $\Lie(G)$ on $\mcO(G)$}
{Action of Dist(G) and Lie(G) on 𝓞(G)}}\label{se:action}

Our next goal is to give the analogue of the classical description of
the Lie algebra of an affine group scheme as the algebra of right
invariant derivations. We shall accomplish this in
Sections~\ref{se:action}--\ref{se:universality}. 

Taking the map $\Delta\colon \mcO(G)\to \mcO(G) \otimes \mcO(G)$,
we compose with 
\[ \mcO(G)\otimes \mcO(G)\to \mcO(G)/\mfm^n \otimes \mcO(G) \]
and take adjoints, to obtain compatible maps $\Dist_n(G)\otimes \mcO(G) \to \mcO(G)$ 
and hence a map
\begin{equation}\label{eq:dist-module}
\mu \colon\Dist(G) \otimes \mcO(G) \to \mcO(G).
\end{equation}
Coassociativity of $\Delta$ shows that this makes 
$\mcO(G)$ a module for the algebra of distributions $\Dist(G)$.

Let 
\[ \beta\colon \Dist(G) \otimes \mcO(G) \to \mcO(G) \]
be adjoint in the above sense to the map
\[ \beta^*=(1-s)\circ\Delta\colon \mcO(G)\to \mcO(G)\otimes \mcO(G). \]

\begin{prop}
We have
$(\beta\otimes 1+
(1\otimes\beta)\circ(s\otimes 1))\circ(1\otimes\Delta)=\Delta\circ\beta$.
\end{prop}
\begin{proof}
The following diagram commutes:
\[ \xymatrix{A \otimes A\otimes A 
\ar@<1ex>[d]^{(1\otimes s)\circ(s\otimes 1)}
&& A \otimes A \ar@<1ex>[d]^s \ar[ll]_(.45){1\otimes\Delta} \\
A \otimes A \otimes A 
\ar@<1ex>[u]^{(s\otimes 1)\circ(1\otimes s)}&&
A \otimes A\ar@<1ex>[u]^s \ar[ll]^(.45){\Delta\otimes 1}} \]
Using the coassociativity of $\Delta$ and the above diagram, we have
\begin{align*}
(\beta^*\otimes 1+(s\otimes 1)\circ(1\otimes \beta^*))\circ \Delta
&= (((1-s)\otimes 1)\circ(\Delta\otimes 1) \\
&\qquad + (s\otimes 1) \circ(1\otimes(1-s))\circ(1\otimes\Delta))\circ\Delta\\
&= ((1\otimes \Delta)-(s\otimes 1)\circ(1\otimes s)\circ(\Delta\otimes 1))\circ\Delta\\
&=((1\otimes\Delta)-(1\otimes\Delta)\circ s)\circ\Delta \\
&=(1\otimes \Delta)\circ(1-s)\circ\Delta\\
&=(1\otimes \Delta)\circ\beta^*.
\end{align*}
Taking adjoints gives 
$(\beta\otimes 1+
(1\otimes\beta)\circ(s\otimes 1))\circ(1\otimes\Delta)=\Delta\circ\beta$.
\end{proof}

This enables us to formulate a Lie action of $\Lie(G)$ on $\mcO(G)$ as follows.
According to Lemma~\ref{le:Delta}, $(1-s)\circ\Delta$ sends 
$\mcO(G)$ to $\mfm\otimes\mfm$. Composing with the map to
$\mfm/\mfm^2\otimes \mcO(G)$, we have a map 
$\beta^*\colon \mcO(G)\to \mfm/\mfm^2\otimes \mcO(G)$. Taking adjoints, 
we obtain a bracket operation
\[ \beta \colon \Lie(G) \otimes \mcO(G) \to \mcO(G). \]
Again using Lemma~\ref{le:Delta}, this action preserves the radical
filtration of $\mcO(G)$.

We wish to relate this to derivations. For this purpose, let us write
$\pi\colon \mcO(G) \to \mcO(G)/\mfm^2$ for the reduction modulo $\mfm^2$.

\begin{lemma}\label{le:picircmu}
We have $\pi\circ\mu=\pi\otimes\eta+\eta\otimes\pi$, where
$\eta\colon \mcO(G)\to \one$ is the augmentation. So
the following diagram commutes.
\[ \xymatrix{\mcO(G)\otimes 
\mcO(G) \ar[r]^(.55)\mu\ar[dr]_{\pi\otimes\eta+\eta\otimes\pi} &
\mcO(G)\ar[d]^\pi \\
& \mcO(G)/\mfm^2} \]
\end{lemma}
\begin{proof}
Clear.
\end{proof}

\section{Derivations}\label{se:derivations}

\begin{defn}\label{def:derivations}
An action of an object $X$ in $\mcCi$ on an algebra $A$ \emph{by derivations}
consists of a bracket operation $\beta\colon X \otimes A \to A$
such that the following diagram commutes.
\[ \xymatrix{X \otimes A \otimes A \ar[r]^(.55){1\otimes \mu}
\ar[d]_{\beta\otimes 1 + (1\otimes \beta)\circ (s\otimes 1)}
& X \otimes A \ar[d]^\beta \\
A \otimes A \ar[r]^\mu & A.} \]

More generally, an $A$-$A$-bimodule $M$ is equipped with left action
$\mu_L\colon A \otimes M \to M$ and right action $\mu_R\colon M\otimes
A \to M$. An action by derivations by $X$ from $A$ to an $A$-$A$-bimodule $M$ is an
operation $\beta\colon X \otimes A \to M$ such that the following
diagram commutes.
\[ \xymatrix{X \otimes A \otimes A \ar[r]^(.55){1\otimes \mu}
\ar[dr]_(.4){\substack{\mu_R\circ(\beta \otimes 1) + \\ \mu_L\circ(1\otimes
  \beta)\circ(s\otimes 1)}} & X \otimes A \ar[d]^\beta \\ & M} \]
\end{defn}

\begin{prop}
The map $\beta\colon \Lie(G) \otimes \mcO(G) \to \mcO(G)$ is an action
by derivations. Thus the
following diagram commutes.
\[ \xymatrix{\Lie(G) \otimes \mcO(G) \otimes \mcO(G) \ar[r]^(.6){1\otimes \mu} 
\ar[d]_{\beta\otimes 1 + (1\otimes \beta) \circ (s\otimes 1)}
& \Lie(G) \otimes \mcO(G)\ar[d]^\beta \\
\mcO(G) \otimes \mcO(G)\ar[r]^\mu & \mcO(G)} \]
\end{prop}
\begin{proof}
We have $(\eta\otimes 1)\circ \Delta=\Id$,
$(1\otimes\eta)\circ \Delta=\Id$ and
\[\Delta\circ\mu=
(\mu\otimes\mu)
\circ(1\otimes s\otimes 1)\circ(\Delta\otimes \Delta). \]
So using Lemma~\ref{le:picircmu}, as maps from $\mcO(G) \otimes \mcO(G)$
to $\mcO(G)/\mfm^2 \otimes \mcO(G)$ we have
\begin{align*}
(\pi\otimes 1)\circ\beta^*\circ\mu
&=(\pi\otimes 1)\circ(1-s)\circ\Delta\circ\mu\\
&=(\pi\otimes 1)\circ(1-s)\circ(\mu\otimes\mu)\circ(1\otimes s\otimes 1)
\circ(\Delta\otimes\Delta)\\
&=(\pi\otimes 1)\circ(\mu\otimes\mu)\circ(1\otimes s\otimes 1)\circ(1-s\otimes s)
\circ(\Delta\otimes\Delta)\\
&=(\pi\otimes\eta\otimes\mu+\eta\otimes\pi\otimes\mu)\circ(1\otimes s\otimes 1)\circ(1-s\otimes s)
\circ(\Delta\otimes\Delta)\\
&=(\pi\otimes\mu)\circ(1\otimes \eta\otimes 1\otimes 1+\eta\otimes 1\otimes 1\otimes 1)\\
&\qquad\circ(1\otimes s \otimes 1)\circ(1-s\otimes s)\circ(\Delta\otimes \Delta)\\
&=(\pi\otimes\mu)\circ(1\otimes 1\otimes \eta\otimes 1+(\eta\otimes 1\otimes 1\otimes 1)\\
&\qquad\circ(1\otimes s\otimes 1))\circ(1-s\otimes s)\circ(\Delta\otimes \Delta)\\
&=(\pi\otimes\mu)\circ(((1-s)\otimes 1)\circ (\Delta \otimes 1)\\
&\qquad+(s\otimes 1)\circ(1\otimes(1-s)\circ(1\otimes\Delta)))\\
&=(\pi\otimes\mu)\circ(\beta^*\otimes 1+(s\otimes 1)\circ(1\otimes\beta^*)).
\end{align*}
Now take adjoints.
\end{proof}

\begin{defn}\label{def:[beta,beta']}
If $\beta\colon X\otimes A\to A$ and $\beta'\colon X' \otimes A\to A$ 
are actions of $X$ and $X'$ on $A$, then we define
the \emph{commutator action} of $X\otimes X'$ on $A$ as
\[ [\beta,\beta']\colon X \otimes X' \otimes A \to A \]
where 
\[ [\beta,\beta']=\beta\circ(1\otimes\beta') -
  \beta'\circ(1\otimes\beta)\circ(s\otimes 1). \]
\end{defn}

\begin{theorem}\label{th:derivations}
If $\beta\colon X\otimes A\to A$ and $\beta'\colon Y\otimes A\to A$ 
are actions of $X$ and $Y$ on $A$ by derivations, then
$[\beta,\beta']\colon X \otimes Y\otimes A\to A$ is an action
by derivations.
\end{theorem}
\begin{proof}
We have
\begin{align*}
[\beta,\beta'] \circ(1\otimes 1\otimes\mu)
&=\beta\circ(1\otimes\beta')\circ(1\otimes 1\otimes\mu)
-\beta'\circ(1\otimes\beta)\circ(s\otimes \mu) \\
&=\beta\circ(1\otimes\mu)\circ((1\otimes \beta'\otimes 1)
+(1\otimes 1\otimes \beta')\circ(1\otimes s\otimes 1))\\
&\hspace{-3mm} -\beta'\circ(1\otimes \mu)\circ((1\otimes \beta \otimes 1)
+(1\otimes 1\otimes \beta)\circ((1\otimes s \otimes 1))
\circ(s\otimes 1 \otimes 1) \\
&=\mu\circ((\beta\otimes 1)+(1\otimes\beta)\circ(s\otimes 1))
\circ((1\otimes\beta'\otimes 1)\\
&\quad+(1\otimes 1\otimes\beta')\circ(1\otimes s\otimes 1))\\
&\quad-\mu\circ((\beta'\otimes 1)+(1\otimes\beta')\circ(s\otimes 1))
\circ((1\otimes\beta\otimes 1)\\
&\quad+(1\otimes 1\otimes\beta)\circ(1\otimes s\otimes 1))
\circ(s\otimes 1 \otimes 1).
\end{align*}
Now we have
\begin{align*}
(1\otimes\beta)\circ(s\otimes 1)\circ(1\otimes\beta'\otimes 1)
&= (\beta'\otimes 1) \circ(1\otimes 1 \otimes \beta)
\circ(1\otimes s \otimes 1)\\
(\beta\otimes 1)\circ(1\otimes 1 \otimes \beta')
\circ(1\otimes s \otimes 1)
&=(1\otimes \beta')\circ(s\otimes 1)\otimes 
(1\otimes\beta\otimes 1)\circ(s\otimes 1\otimes 1)
\end{align*}
and so terms cancel in the above expression to give
\begin{align*}
\qquad&=\mu\circ((\beta\otimes 1)\circ(1\otimes\beta'\otimes 1)
+(1\otimes\beta)\circ(1\otimes 1\otimes \beta')
\circ(s\otimes 1\otimes 1)\circ(1\otimes s\otimes 1)\\
&\qquad -(\beta'\otimes 1)\circ(1\otimes\beta\otimes 1)
\circ(s\otimes 1\otimes 1) - (1\otimes\beta')\circ
(1\otimes 1\otimes \beta))\\
&=\mu\circ([\beta,\beta']\otimes 1 +
(1\otimes [\beta,\beta'])\circ(s\otimes 1\otimes 1)
\circ(1\otimes s\otimes 1)).
\end{align*}
Thus $[\beta,\beta']$ is an action by derivations.
\end{proof}

\begin{defn}
We say that $\beta_X\colon X \otimes \mcO(G) \to \mcO(G)$ is a \emph{right invariant} action 
if the following diagram commutes:
\[ \xymatrix{X \otimes \mcO(G) \ar[r]^{\beta_X}\ar[d]_{1\otimes\Delta} &
    \mcO(G)\ar[d]^\Delta \\
    X \otimes \mcO(G) \otimes \mcO(G) \ar[r]_(.55){\beta_X\otimes 1} & \mcO(G) \otimes
    \mcO(G).} \]
\end{defn}

\begin{rk}
We use right invariant actions rather than left invariant actions,
because we are interested in the case $X=\Lie(G)$, and we want 
this to act on $\mcO(G)$ on the left.
\end{rk}

\begin{lemma}\label{le:etabeta}
If $\beta_X\colon X\otimes \mcO(G) \to \mcO(G)$ is right invariant then
$\beta_X$ can be recovered from $\eta\circ\beta_X$ via the formula
\[ \beta_X=((\eta\circ\beta_X)\otimes 1)\circ(1\otimes\Delta). \]
\end{lemma}
\begin{proof}
This follows from the commutative diagram
\begin{equation*} 
\xymatrix{X \otimes \mcO(G) \ar[r]^{\beta_X}\ar[d]_{1\otimes\Delta} & \mcO(G)\ar[d]_\Delta\ar[dr]^{\Id} \\
X \otimes \mcO(G) \otimes \mcO(G)\ar[r]_(.55){\beta_X\otimes 1}&\mcO(G)\otimes \mcO(G)
\ar[r]_(.6){\eta\otimes 1} & \mcO(G)}
\end{equation*}
Here, $\eta\colon \mcO(G) \to \one$ is the counit.
\end{proof}

\begin{lemma}
If $\beta_X\colon X\otimes \mcO(G) \to \mcO(G)$ and $\beta_Y\colon Y\otimes
\mcO(G) \to \mcO(G)$ are actions by right invariant derivations then so is
the commutator action $[\beta_X,\beta_Y]\colon X \otimes Y \otimes \mcO(G)
\to \mcO(G)$.
\end{lemma}
\begin{proof}
For right invariance, we have
\begin{align*}
([\beta_X,\beta_Y]\otimes 1)\circ(1\otimes 1\otimes\Delta) &=
\Bigl(\bigl(\beta_X\circ(1\otimes\beta_Y)-\beta_Y\circ(1\otimes\beta_X)
\circ(s\otimes 1)\bigr)\otimes 1\Bigr)\\
&\qquad\qquad\circ(1\otimes 1\otimes\Delta) \\
&=(\beta_X\otimes 1)\circ(1\otimes\beta_Y\otimes 1)\circ
(1\otimes 1\otimes \Delta)  \\
&\qquad - (\beta_Y\otimes 1)\circ(1\otimes\beta_X\otimes 1)\circ
(1\otimes 1\otimes\Delta)\circ(s\otimes 1)\\
&=(\beta_X\otimes 1)\circ(1\otimes\Delta)\circ(1\otimes\beta_Y)\\
&\qquad- (\beta_Y\otimes 1)\circ(1\otimes\Delta)\circ
(1\otimes\beta_X)\circ(s\otimes 1)\\
&=\Delta\circ\beta_X\circ(1\otimes\beta_Y)
- \Delta\circ\beta_Y\circ(1\otimes\beta_X)\circ(s\otimes 1)\\
&=\Delta\circ[\beta_X,\beta_Y].
\end{align*}
This, together with Theorem~\ref{th:derivations}, proves the lemma.
\end{proof}

\section{\texorpdfstring{Universality of $\Lie(G)$}
{Universality of Lie(G)}}\label{se:universality}

In this section, we show that $\Lie(G)$ is universal among objects
acting on $\mcO(G)$ by right invariant derivations.

\begin{lemma}\label{le:m2vanish}
If $\beta_X\colon X \otimes \mcO(G) \to \mcO(G)$ is an action by
derivations then $\eta\circ\beta_X$ vanishes on $X\otimes \mfm^2$.
\end{lemma}
\begin{proof}
Consider the commutative diagram
\[ \xymatrix{X \otimes \mcO(G) \otimes \mcO(G) \ar[r]^(.6){1\otimes \mu}
\ar[d]_{\beta_X\otimes 1 + (1\otimes \beta_X)\circ (s\otimes 1)}
& X \otimes \mcO(G) \ar[d]_{\beta_X}\ar[dr]^(.55){\eta\circ\beta_X} \\
\mcO(G) \otimes \mcO(G) \ar[r]^(.55)\mu & \mcO(G) \ar[r]^\eta & \one.} \]
The lower composite  $\eta\circ\mu$ vanishes on $\mfm \otimes \mcO(G)
\oplus \mcO(G) \otimes \mfm$. Consider the left vertical map. 
The map $\beta_X \otimes 1$ sends $X
\otimes \mfm \otimes \mfm$ into $\mcO(G) \otimes \mfm$, and the map
$(1\otimes \beta_X)\circ(s\otimes 1)$ sends $X \otimes \mfm \otimes
\mfm$ into $\mfm \otimes \mcO(G)$. It follow that the composite from the
top left to the bottom right is zero. The image of the top map
$1\otimes\mu$ on $X \otimes \mfm \otimes \mfm$ is $X \otimes \mfm^2$,
and hence $\eta\circ\beta_X$ sends $X\otimes \mfm^2$ to zero. 
\end{proof}

\begin{prop}
The action $\beta\colon\Lie(G)\otimes \mcO(G)\to \mcO(G)$ is by right invariant derivations.
\end{prop}
\begin{proof}
The right invariance is adjoint to the associativity of comultiplication on $\mcO(G)$.
\end{proof}

\begin{theorem}\label{th:Lie-derivations}
The action of $\Lie(G)$ on $\mcO(G)$ by right invariant derivations is
universal, in the sense that given an action $\beta_X\colon X
\otimes \mcO(G) \to \mcO(G)$ by right
invariant derivations, there exists a unique map $\beta'_X\colon X \to
\Lie(G)$ such that the action is given by $\beta\circ(\beta'_X \otimes
1)$:
\[ \xymatrix{X \otimes \mcO(G) \ar[dr]^(.6){\beta_X}\ar[d]_{\beta'_X\otimes 1} \\
\Lie(G) \otimes \mcO(G) \ar[r]_(.62)\beta & \mcO(G).} \]
\end{theorem}
\begin{proof}
Given a right invariant action by derivations $\beta_X\colon X \otimes
\mcO(G)\to \mcO(G)$, by Lemma~\ref{le:m2vanish} $\eta\circ\beta_X$ vanishes on $X
\otimes \mfm^2$, and hence induces a map $X \otimes \mfm/\mfm^2 \to
\one$. We define $\beta'_X$ to be the adjoint map $X \to
(\mfm/\mfm^2)^*=\Lie(G)$. Applying this to the action of $\Lie(G)$
gives us the commutative diagram
\[ \xymatrix{\Lie(G) \otimes \mfm \ar[r]^(.6)\beta \ar[d] & \mcO(G)\ar[d]^\eta \\
\Lie(G) \otimes \mfm/\mfm^2 \ar[r]^(.7)\ev & \one,} \]
where $\ev$ is the evaluation map.
By construction, the composite
\[ X \otimes \mfm \xrightarrow{\beta'_X\otimes 1} \Lie(G) \otimes \mfm
\to \Lie(G) \otimes \mfm/\mfm^2 \xrightarrow{\ev} \one \]
is equal to $\eta\circ\beta_X$. Thus 
$\eta\circ\beta\circ(\beta'_X\otimes 1) = \eta\circ\beta_X$. Since
$\beta\circ(\beta'_X\otimes 1)$ is also an action by right invariant
derivations, applying Lemma~\ref{le:etabeta} we have 
$\beta\circ(\beta'_X\otimes 1)=\beta_X$.
\end{proof}

\section{Restricted Lie algebras}\label{se:restricted} 

In this section, we show that $\Lie(G)$ is a restricted Lie
algebra. For this purpose, we recall the definition from
Fresse~\cite{Fresse:1999a}. First recall that for the 
associative operad $\As$, the vector space $\As(n)$ has a basis
consisting of monomials $X_{\sigma(1)}\dots X_{\sigma(n)}$ for $\sigma\in\Sigma_n$,
with $\Sigma_n$ acting on the subscripts. This is a free module of
rank one over $k\Sigma_n$. Then the 
Lie operad $\Lieop$
is the suboperad of the associative operad $\As$ spanned by repeated
commutators. A basis for $\Lieop(n)$ is given by the expressions
\[ [[..[X_{\sigma(1)},X_{\sigma(2)}],\dots],X_{\sigma(n)}] \]
where $\sigma$ runs over the permutations with $\sigma(1)=1$. As a
submodule of the regular representation, this can be written as the
left ideal $k\Sigma_n.\omega_n$ generated by
\begin{equation}\label{eq:omega}
\omega_n=(1-c_2)(1-c_3)\dots(1-c_n) 
\end{equation}
where $c_i$ is the $i$-cycle $(i\ i-1\ \dots\ 2\ 1)$. Thus
$\Lieop(n)\cong k\Sigma_n/\Ann(\omega_n)$, where
\[ \Ann(\omega_n)=\{x\in k\Sigma_n\mid x\omega_n=0\}. \]

Now given an object $V$ in $\mcCi$, the free operadic Lie algebra on
$V$ is $\bigoplus_n (\Lieop(n) \otimes V^{\otimes n})_{\Sigma_n}$, the
coinvariants of the symmetric groups. The structure of an operadic Lie
algebra on $V$ in
$\mcCi$ is then given by maps 
$\bigoplus_n(\Lieop(n)\otimes V^{\otimes n})_{\Sigma_n}\to V$ respecting the
operadic composition. 

We define a restricted Lie algebra as follows.

\begin{defn} [Fresse~\cite{Fresse:1999a}]\label{def:Fresse}
We define $\Gamma(\Lieop,V)$ to be $\bigoplus_n(\Lieop(n)\otimes
V^{\otimes n})^{\Sigma_n}$, the invariants rather than the
coinvariants, and put a suitable monad structure on this functor. 
Then the structure of 
 a $\Gamma\Lieop$-algebra, or \emph{restricted Lie algebra}, on $V$ in $\mcCi$ is a map
 $\Gamma(\Lieop,V)\to V$ making $V$ an algebra over this monad. Thus
 we are given maps 
\[ \phi_n\colon (\Lieop(n)\otimes V^{\otimes n})^{\Sigma_n} \to V \]
satisfying compatibility with composing operations. 
\end{defn}

For example, if $\mcCi=\VV$, then $\phi_2$ takes
the invariant element $[X_1,X_2] \otimes (v\otimes w-w\otimes v)$ to
$[v,w]$ to define a Lie bracket. In characteristic $p$, the element
\[ X_{\Sigma_p}=\sum_{\sigma\in\Sigma_p}X_{\sigma(1)}X_{\sigma(2)}\dots X_{\sigma(p)}
=\sum_{\substack{\sigma\in\Sigma_p\\\sigma(1)=1}}[..[X_{\sigma(1)},X_{\sigma(2)}],..,X_{\sigma(p)}] \]
is invariant, and in vector spaces, $\phi_p$
takes $X_{\Sigma_p}\otimes v^{\otimes p}$ to the $p$-restricted power
$v^{[p]}$.

If we make the corresponding definition for associative algebras, we
find that $\Gamma\As$-algebras are equivalent to operadic
$\As$-algebras via the transfer map. Namely, the map
\[ \Tr\colon (\As\otimes V^{\otimes n})_{\Sigma_n}\to
(\As\otimes V^{\otimes n})^{\Sigma_n} \]
induced by $\sum_{\sigma\in\Sigma_n}\sigma\in k\Sigma_n$ is an
isomorphism, commuting with the identification of each side with
$V^{\otimes n}$. The monad structure on $\Gamma\As$ is arranged so
that $\Gamma\As$-algebras are equivalent to $\As$-algebras.

Now the inclusion of operads $\Lieop\subseteq\As$ induces a diagram
\[ \xymatrix{(\Lieop(n)\otimes V^{\otimes n})_{\Sigma_n}\ar[r]\ar[d]^{\Tr}
    &(\As(n)\otimes V^{\otimes
      n})_{\Sigma_n}\ar[d]^{\Tr}_\cong\ar[r]^(.7)\cong&V^{\otimes n}\ar[d]^=\\
(\Lieop(n)\otimes V^{\otimes n})^{\Sigma_n}\ar@{^(->}[r]&
(\As(n)\otimes V^{\otimes n})^{\Sigma_n}\ar[r]^(.7)\cong&V^{\otimes n}.} \]
This diagram shows that the kernel of the top horizontal map is the same as the
kernel of the left vertical map. According to Section~4 of
Etingof~\cite{Etingof:2018a}, the kernel of the top horizontal map
consists of the identities that we need to add in order for an
operadic Lie algebra to be a Lie algebra. In other words, an operadic
Lie algebra is a Lie algebra if and only if the kernel of the top
horizontal map is in the kernel of the operadic Lie structure map
$(\Lieop(n)\otimes V^{\otimes n})_{\Sigma_n}\to V$ for each $n$.
Thus (see also Definition~7.10 of~\cite{Etingof:2018a}) we have a diagram of functors
\begin{equation}\label{eq:functors} 
\vcenter{\xymatrix@R=5mm{\text{\sf $\Gamma$-associative algebras} \ar[r] \ar[dd]^{\Tr^*}_\simeq &
\text{\sf restricted Lie algebras} \ar[dr]\ar[dd]^{\Tr^*} \\
&&\text{\sf Lie algebras} \ar[dl] \\
\text{\sf associative algebras} \ar[r]\ar@{-->}[uur] &
\text{\sf operadic Lie algebras}} }
\end{equation}
We write $\Lie(A)$ for the image of an associative algebra under
composite functor to Lie algebras.
The following proposition summarises the above discussion.

\begin{prop}\label{pr:forget}
The functor from restricted Lie algebras in a symmetric
tensor category $\mcCi$ to operadic Lie
algebras in $\mcCi$ induced by $\Tr$ factors through Lie algebras in $\mcCi$. 
Moreover, the forgetful functor from associative algebras in $\mcCi$ to operadic
Lie algebras in $\mcCi$ factors through restricted Lie algebras in $\mcCi$.\qed
\end{prop}

Applying this to the associative distribution algebra $\Dist(G)$ for a
group scheme $G$ in $\mcC$, we obtain a restricted Lie algebra
structure. We wish to show that the Lie algebra $\Lie(G)=\Gr\Dist_1(G)$
inherits a restricted Lie algebra structure this way. 

Let us write $\Delta^{n-1}$ for the composite of $n-1$ comultiplications
\[  \Delta\circ(\Delta\otimes\Id)\circ\dots\circ(\Delta\otimes\dots\otimes\Id)
\colon \mcO(G)\to \mcO(G)^{\otimes n} . \]

\begin{lemma}
The composite of the action of the element $\omega_n$
of~\eqref{eq:omega} with $\Delta^{n-1}$, 
\[ \mcO(G) \xrightarrow{\Delta^{n-1}} \mcO(G)^{\otimes n}
  \xrightarrow{\omega_n} \mcO(G)^{\otimes n} \]
sends $\mfm$ into $\mfm^{\otimes n}$ and 
$\mfm^j$ into 
\[ \sum_{\substack{i_1+\dots +i_n=n+j-1\\
      \text{\rm each } i_j\ge 1}}\mfm^{i_1}\otimes\dots\mfm^{i_n}. \]
\end{lemma}
\begin{proof}
The map $\omega_n\circ\Delta^{n-1}$ can be rewritten as a composite
\begin{multline*} 
\mcO(G) \xrightarrow{\Delta}\mcO(G)^{\otimes 2}
\xrightarrow{(1-s)}\mcO(G)^{\otimes 2} 
\xrightarrow{\Delta\otimes\Id}\mcO(G)^{\otimes 3}
\xrightarrow{(1-s)\otimes\Id}\mcO(G)^{\otimes 3}
\xrightarrow{\Delta\otimes\Id^{\otimes 2}}\cdots\\
\cdots
\xrightarrow{\Delta\otimes\Id^{\otimes(n-2)}}\mcO(G)^{\otimes n}
\xrightarrow{(1-s)\otimes\Id^{\otimes(n-2)}}\mcO(G)^{\otimes n}. 
\end{multline*}
Now apply Lemma~\ref{le:Delta} and induction on $n$.
\end{proof}

Dualising, we get the following analogue of Proposition~II.7.2.3 of
Demazure and Gabriel~\cite{Demazure/Gabriel:1970a}.

\begin{theorem}
Let $G$ be a group scheme in $\mcCi$. Then
the map 
\[ (\Lieop(n)\otimes\Dist(G)^{\otimes n})^{\Sigma_n}\to
\Dist(G) \] 
induced by the algebra structure on $\Dist(G)$ sends
\[ \left(\Lieop(n)\otimes\bigoplus_{i_n+\dots+i_n=j}
\Dist^+_{i_1}(G)\otimes\dots\otimes\Dist^+_{i_n}(G)\right)^{\Sigma_n} \] 
into $\Dist^+_{j-n+1}(G)$. In particular, restricting to
$\Lie(G)=\Dist^+_1(G)$, it sends
\[ (\Lieop(n)\otimes\Lie(G)^{\otimes n})^{\Sigma_n} \to \Lie(G), \]
making $\Lie(G)$ a restricted Lie algebra in $\mcCi$.\qed
\end{theorem}


\part{\texorpdfstring{\centering The case $\mcC=\CV$}
{The case C=Ver₄⁺}}

We now make explicit the abstract theory presented above in the case
of the symmetric tensor category $\CV$. We write $\CVi$ for
the ind-completion of $\CV$.

Let $k$ be a field of characteristic two.
An object in $\CVi=\CVi(k)$ consists of a vector space $V$ together with a 
map $d\colon V \to V$ satisfying $d^2=0$. The tensor product of two
such, $(V,d)$ and $(V',d')$, is $(V\otimes V',d\otimes 1 + 1 \otimes d')$.
The symmetric braiding
\[ s\colon V \otimes V' \to V'\otimes V \]
is given by 
\[ s(v\otimes v') = v'\otimes v + dv' \otimes dv. \]

The category $\CV$ has just one simple object, which is the field
$k$. This is the tensor identity $\one$ in $\CV$.
There are two isomorphism classes of indecomposable objects in $\CV$,
namely the simple $k$ and its projective cover $P$. There is a short
exact sequence
\[ 0 \to k \to P \to k \to 0. \]

\section{Algebras}

This section expands on Section~\ref{se:CA} in the context of $\CV$.
An algebra in $\CVi$ is an object $A$ in $\CVi$ together
with a morphism $\mu\colon A \otimes A \to A$ satisfying the usual associativity
and unit conditions; namely a differential algebra over $k$. This consists of an 
associative algebra $A$ in the usual sense, with a linear map $d\colon A\to A$ satisfying
$d^2=0$ and $d(xy)=dx.y+x.dy$ for all $x$, $y\in A$. In particular, this implies
that $d(1)=0$.

The tensor product $A\otimes A'$ of two algebras $A$, $A'$ in $\CVi$
is given by taking the tensor product of the vector spaces $A \otimes_k A'$
with differential $d(x\otimes y) = dx \otimes y + x \otimes dy$
and multiplication 
\[ (x \otimes y)(x' \otimes y') = xx' \otimes yy' + x.dx' \otimes dy.y'. \]

The definition of commutative algebra given in Section~\ref{se:CA}
takes the following form in $\CVi$.

\begin{defn}\label{def:ep}
A \emph{commutative algebra} over $\CVi$ consists of a differential algebra $A$ satisfying
the identity
\begin{equation}\label{eq:comm}
xy + yx = dx.dy
\end{equation}
for all $x$, $y\in A$. 

The commutative algebras over $\CVi$ form a category over $\CVi$
which we denote $\CVAlg$. If $A$ and $B$ are commutative algebras
over $\CVi$ then so is $A\otimes B$, with multiplication
\[ (x\otimes y) (x'\otimes y')=xx'\otimes yy' + x.dx' \otimes dy.y'. \]
There is a full embedding of the category $\CAlg=\CAlg_{\VV(k)}$ 
into $\CVAlg$. It has a left adjoint $u\colon \CVAlg \to \CAlg$
which takes a commutative algebra $A$ over $\CVi$ to
the quotient $A_u$ of $A$ by the linear span of elements of the form $x.dy$. 
This is an ideal, so the quotient is an object in $\CAlg$.
\end{defn}

\begin{eg}\label{eg:K}
Mod two topological K-theory is a contravariant functor from the
homotopy category of CW-complexes to the category
$\CVAlg$, with the Bockstein map as the differential. 
Actually it admits two different multiplications,
but one is the reverse of the other, so we make a choice.
The two choices are related by an automorphism of $\CV$,
so it does not matter which choice we make. This is investigated in
papers of Araki and coauthors~\cite{Araki:1967a,Araki/Toda:1965a, 
Araki/Toda:1966a,Araki/Yosimura:1970a,Araki/Yosimura:1971a}.

Morava K-theory $K(n)$ in characteristic two, as well as the $BP$-module
spectrum $P(n)$ have similar properties. See 
W\"urgler~\cite[Remark~7.3]{Wurgler:1977a}, \cite[Proposition~2.4]{Wurgler:1986a},
Mironov~\cite{Mironov:1979a},
Kultze and W\"urgler~\cite{Kultze/Wurgler:1987a,Kultze/Wurgler:1988a},
Kane~\cite[\S14]{Kane:1988a},
Nassau~\cite{Nassau:2002a}, 
Strickland~\cite[Theorem~2.13]{Strickland:1999a}.
\end{eg}

\begin{lemma}\label{le:dx2}
If $x$, $y$ are elements of a commutative algebra over $\CV$ then
$(dx)^2=0$ and $x.dy=dy.x$.
\end{lemma}
\begin{proof}
This follows from \eqref{eq:comm} by replacing $y$ by $x$, respectively $dy$.
\end{proof}

An object in $\CV$ can be written as a direct sum of $m$ copies of $k$
and $n$ copies of $P$. This has a $k$-basis of the form
\begin{equation}\label{eq:basis}
x_1,\dots,x_{m+n},w_1,\dots,w_n 
\end{equation}
with 
\[ d(x_i)=\begin{cases} w_i & 1\le i \le n \\ 0 & n+1\le i \le
  m \end{cases}\qquad
d(w_i) = 0. \] 
We write $V_{m+n|n}$ for this object. These form a complete set of
representatives for the isomorphism classes of 
objects in $\CV$, with $m,n\ge 0$.

\begin{lemma}\label{le:free}
The \emph{free commutative algebra} in $\CVi$ on $V_{m+n|n}$
has generators as an algebra 
\[ x_1,\dots,x_{n+m},w_1,\dots,w_n,\] 
with 
\[
d(x_i)=\begin{cases} w_i & 1\le i \le n \\ 0 & n+1\le i \le
  m \end{cases},\qquad
d(w_i) = 0 \]
and relations
\[ w_i^2=0,\qquad 
x_iw_j=w_jx_i,\qquad
x_ix_j+x_jx_i=w_iw_j=w_jw_i. \] 
\end{lemma}
\begin{proof}
Given these generators, defining $w_i$ to be $d(x_i)$, the given 
relations follow from Lemma~\ref{le:dx2}. Conversely, the differential algebra
defined by these generators and relations satisfies the commutativity
condition \eqref{eq:comm}.
\end{proof}

\begin{defn}
For notation, we shall write $k\lb x_1,\dots,x_n\rb$ for the free commutative
algebra in $\CV$ on $x_1,\dots,x_n$, as described in Lemma~\ref{le:free}.
\end{defn}

\begin{eg}\label{eg:xx^-1}
Let's examine $k\lb x,y\rb/(xy+1)$. We have $d(xy)=d(1)=0$, 
so $dx.y+x.dy=0$. Multiplying by $y$, we see that $dy=y^2dx$.
Writing $y$ as $x^{-1}$, we have $d(x^{-1})=x^{-2}dx$. Thus
as an ordinary algebra we have $k\lb x,y\rb/(xy)=k[x,x^{-1},w]/(w^2)$
with $dx=w$, $d(x^{-1})=x^{-2}w$, $dw=0$. We shall write $k\lb x,x^{-1}\rb$ 
for this object in $\CVAlg$.
\end{eg}

\begin{lemma}\label{le:inverse}
Let $A$ be a commutative algebra in $\CVi$, and let $u$ and
$v$ be elements of $A$ satisfying $uv=1$. Then also $vu=1$.
\end{lemma}
\begin{proof}
Let $I$ be the ideal in $A$ generated by the elements $dx$ for $x \in A$.
Then $I^2=0$, and $A/I$ is commutative in the ordinary sense. These are
the hypotheses we use to prove the lemma. If $uv=1$ then $vu$ and $1$
have the same image in $A/I$ and so $vu-1\in I$. Thus $(1-vu)^2=0$.
Expanding this out, we have $1-2vu+vuvu=0$. Using $uv=1$, this becomes
$1-vu=0$ and so $vu=1$.
\end{proof}

\begin{rk}
Essentially the same 
proof shows more generally that if $I$ is a nilpotent ideal in a
ring $A$ with $A/I$ commutative (in the ordinary sense) and
$u$ and $v$ are elements with $uv=1$ then also $vu=1$.
The proof is that $vu$ is idempotent, so $1-vu$ is also
idempotent. But it is also nilpotent, and hence it is zero.
\end{rk}

\section{Affine schemes}

This section expands on Section~\ref{se:tangent} in the context of $\CV$.
An affine scheme $S$ in $\CVi$ is a representable functor
\[ S\colon \CVAlg\to \Set \]
with representing object the coordinate ring $\mcO(S)$. 

The canonical algebra $E$ for $\CV$ is a two dimensional
space with basis $\ep, d\ep$. So the ring of dual numbers over $\CV$
(Definition~\ref{def:dual-numbers}) is 
\[ \mcE = k[\ep,d\ep]/(\ep^2,\ep\,d\ep,d\ep^2). \]
The differential $d$ takes $1$ and $d\ep$ to zero, and $\ep$ to $d\ep$.
The ring $\mcE$ comes with an augmentation map $\eta\colon\mcE\to k$ which
takes $1$ to $1$ and takes $\ep$ and $d\ep$ to zero.

Let $x\colon\mcO(S) \to k$ have kernel $\mfm_x$.
The tangent space to $S$ at $x$ is 
\[ \Tan_{S,x}=(\mfm_x/\mfm_x^2)^* \cong 
\Hom_\CV(\mfm/\mfm_x^2,E)\cong \Hom_{\CVAlg}(\mcO(S),\mcE)_x, \]
where $\Hom_{\CVAlg}(\mcO(S),\mcE)_x$ is the
subspace of $\Hom_{\CVAlg}(\mcO(S),\mcE)$ consisting
of those morphisms whose composite with $\eta$ is
equal to $x$. The isomorphism between $(\mfm_x/\mfm_x^2)^*$ 
and $\Hom_{\CVAlg}(\mcO(S),\mcE)_x$ is described explicitly 
in the following lemma, which shows that tangent vectors can
be thought of as $x$-derivations from $\mcO(S)$ to $k$ in
the ordinary sense of derivation.

\begin{lemma}\label{le:f:k[S]to-k-ep}
If $f\in\Hom_{\CVAlg}(\mcO(S),\mcE)_x$, then
$f$ takes the form 
\begin{equation}\label{eq:f(a)} 
f(a)=x(a)+f'(da)\ep +f'(a)d\ep 
\end{equation}
where $f'\colon \mcO(S)\to k$ is a linear map which takes value
zero on $1$ and on $\mfm_x^2$, and satisfies
\begin{equation}\label{eq:f'(ab)} 
f'(ab)=f'(a)x(b)+x(a)f'(b). 
\end{equation}
Conversely, if $f'$ is given, satisfying these conditions, then
the map $f$ defined by~\eqref{eq:f(a)} is an element of
$\Hom_{\CVAlg}(\mcO(S),\mcE)_x$.
\end{lemma}
\begin{proof}
If the composite of $f$ with $\eta$ is $x$, then $f$ takes
the form
\[ f(a) = x(a) + g(a)\ep + h(a)d\ep \]
for suitable linear maps $g$, $h$ from $\mcO(S)$ to $k$.
Applying $d$ to this equation, we obtain
\[ d(f(a))=g(a)d\ep. \]
On the other hand, since $x(da)=0$, we have
\[ f(da) = g(da)\ep + h(da)d\ep. \]
Since $f$ is a morphism in $\CVi$, it commutes with the action of $d$,
so these expressions have to be equal. Hence we have
\[ g(da)=0, \qquad h(da)=g(a). \]
The second of these says that $h$ determines $g$, and the first
is then automatic. So we write $f'(a)$ for $h(a)$
and $f'(da)$ for $g(a)$.

The statement that $f$ is an algebra homomorphism
translates as
\[ x(ab) + f'(d(ab))\ep + f'(ab)d\ep = (x(a)+f'(da)\ep+f'(a)d\ep)(x(b)+f'(db)\ep+f'(b)d\ep). \]
Multiplying this out, and using the fact that $x(ab)=x(a)x(b)$, this is equivalent to
\[ (f'(da.b)+f'(a.db))\ep + f'(ab)d\ep = (x(a)f'(db)+f'(da)x(b))\ep + (x(a)f'(b)+f'(a)x(b))d\ep. \]
Comparing coefficients of $\ep$ and $d\ep$, this is equivalent to two equations:
\begin{align*}
f'(da.b)+f'(a.db) &= x(a)f'(db)+f'(da)x(b), \\
f'(ab)&=x(a)f'(b) + f'(a)x(b).
\end{align*}
The second of these implies the first, and so $f$ is an algebra homomorphism
if and only if $f'$ satisfies~\eqref{eq:f'(ab)}. If this holds, then $f'$ vanishes on
$1$ and on $\mfm_x^2$.
\end{proof}

\section{Group schemes}

This section expands on Section~\ref{se:group-schemes} in the context
of $\CV$.
An affine group scheme $G$ over $\CVi$ is a functor
\[ G\colon \CVAlg \to \Grp \]
such that composing with the forgetful functor $\Grp\to\Set$
gives a representable functor.
If $G$ is an affine group scheme over $\CVi$ then the
coordinate ring $\mcO(G)$ is a Hopf algebra over $\CVi$.

\begin{warning}
Just as in the case of Hopf superalgebras, 
a Hopf algebra over $\CVi$ is not a Hopf algebra in the
ordinary sense, because the comultiplication is not a map of algebras 
when the tensor product is given the usual algebra structure.
Hopf algebras in $\CVi$ are given the cumbersome name 
\emph{differential near Hopf algebras} in~\cite{Araki:1967a}.
\end{warning}

\begin{eg}
Continuing with Example~\ref{eg:K},
the mod two K-theory of a compact Lie group is
a commutative Hopf algebra over $\CVi$, and hence
gives rise to an affine group scheme.
This example is discussed at length by Araki and his
coauthors in \cite{Araki:1967a,Araki/Toda:1965a,Araki/Toda:1966a,
Araki/Yosimura:1970a,Araki/Yosimura:1971a}.
\end{eg}

\begin{eg}\label{eg:Ga}
The \emph{additive group scheme} $\widetilde\bbG_a$ is the functor
assigning to each object in $\CVAlg$ its additive group. It
has $\mcO(\widetilde\bbG_a)=k\lb x\rb =k[x,dx]/((dx)^2)$
with coproduct $\Delta(x)=x\otimes 1 + 1 \otimes x$,
$\Delta(dx)=dx\otimes 1 + 1 \otimes dx$, counit
$\eta(x)=0$, $\eta(dx)=0$, antipode $\tau(x)=x$, $\tau(dx)=dx$.
\end{eg}

\begin{eg}\label{eg:Gm}
The \emph{multiplicative group scheme} $\widetilde\bbG_m$ is the
functor assigning to each object in $\CVAlg$ its multiplicative
group of invertible elements. Note that an element of such
an object is invertible if and only if it is invertible modulo
the image of $d$. We have
\[ k[\widetilde\bbG_m]=k\lb x,x^{-1}\rb=k[x,x^{-1},dx]/((dx)^2) \]
(see Example~\ref{eg:xx^-1}) with coproduct
\[ \Delta(x) = x \otimes x,\qquad
\Delta(x^{-1}) = x^{-1}\otimes x^{-1},\qquad 
\Delta(dx)=dx\otimes x + x \otimes dx, \]
counit $\eta(x)=1$, $\eta(x^{-1})=1$, $\eta(dx)=0$, antipode
$\tau(x)=x^{-1}$, $\tau(x^{-1})=x$, $\tau(dx)=x^{-2}dx$.
\end{eg}

\begin{rk}
Our notation $\widetilde\bbG_a$ and $\widetilde\bbG_m$ corresponds to the
notation $\bbG'_a$ and $\bbG'_m$ in Hu~\cite{Hu:GLVer4+}.
\end{rk}

\begin{eg}\label{eg:GL}
Let $V_{m+n|n}$ be a direct sum of $m$ copies of $k$ and $n$ copies of $P$ in
$\CV$. Then
the \emph{general linear group scheme} $\GL(V_{m+n|n})=\GL(m+n|n)$ is the functor
assigning to each object $A$ in $\CVAlg$ the multiplicative group
of automorphisms of the $A$-module $V_{m+n|n}\otimes A$. 
Again note that
a matrix is invertible if and only if it is invertible modulo the
image of $d$. We have
\begin{multline*} 
k[\GL(m+n|n)]=k[ x_{i,j}\ (1\le i,j\le m+n), dx_{i,j}\ (1\le i,j\le
  n),u,du]\\
/((dx_{i,j})^2, (du)^2, u.\det(x_{i,j})+1). 
\end{multline*}
(where $dx_{i,j}=0$ if either $i$ or $j$ is larger than $n$).
Note that by Lemma~\ref{le:inverse}, in $k[\GL(m+n|n)]$ we have
\[ \det(x_{i,j}).u=u.\det(x_{i,j})=1. \]
Note also that we may take any order for the variables $x_{i,j}$ for the 
definition of a determinant here, but changing the order changes
what $u$ is.
For example, if $n=2$, $m=0$, and
\[ u(x_{1,1}.x_{2,2}+x_{1,2}.x_{2,1}) = 1 \]
then 
\[ (u+u^2dx_{1,1}.dx_{2,2})(x_{2,2}.x_{1,1}+x_{1,2}.x_{2,1}) = 1. \]
In general, the determinant is well defined modulo the
ideal generated by the image of $d$, but adding an element
of this ideal does not affect the effect of adjoining the inverse 
of the determinant.
For definiteness we use the lexicographic order
on the $x_{i,j}$. 
Using the same lexicographic order to define the adjoint of a matrix $X$
of elements of an object in $\CVAlg$,
we don't have $X.\Adj(X)=\det(X).I$, but the difference squares
to zero, so we have 
\[ X. \Adj(X).u.X.\Adj(X).u=I. \]
So the inverse of $X$ is $\Adj(X).u.X.\Adj(X).u$.

The coproduct on $k[\GL(n+m|m)]$ is given by
\[ \Delta(x_{i,j})=\sum_{\ell}x_{i,\ell}\otimes
  x_{\ell,j},\quad \Delta(u)=u\otimes u. \]
The counit is $\ep(x_{i,j})=\delta_{i,j}$ (Kronecker's delta),
$\ep(dx_{i,j})=0$, $\ep(u)=1$.
The antipode $\tau(x_{i,j})$ is given by the matrix entries in
$\Adj(X).u.X.\Adj(X).u$, which is equivalent to $\Adj(x_{i,j})$ modulo the
ideal generated by the image of $d$, and $\tau(dx_{i,j})=d\tau(x_{i,j})$
then simplifies to $d(\Adj(x_{i,j}))$.
\end{eg}

\begin{lemma}
If $G$ is an affine group scheme over $\CVi$ then $\mcO(G)_u$
(see Definition~\ref{def:ep}) is a
commutative Hopf algebra over $\VV$, defining an affine group
scheme over $\VV$ denoted $G_u$.
\end{lemma}
\begin{proof}
We have
\begin{align*}
\Delta(x.dy)&=\Delta(x).d(\Delta(y)) \\
&=\sum(x' \otimes x'').\sum(dy'\otimes y''+y'\otimes dy'') \\
&=\sum(x'.dy' \otimes x''y'' + x'y' \otimes x''.dy'').
\qedhere
\end{align*}
\end{proof}

Applying this functor to the group schemes $\widetilde\bbG_a$ and
$\widetilde\bbG_m$  above gives the corresponding more familiar group schemes
$\bbG_a$ and $\bbG_m$ over $\VV$.

\begin{defn}
A \emph{finite group scheme} $G$ in $\CV$ is an affine group scheme 
in $\CVi$ whose coordinate ring $\mcO(G)$ is in $\CV$, which is to say
that it is finite dimensional over $k$.
\end{defn}

If $G$ is a finite group scheme in $\CV$ then the dual Hopf algebra
\[ kG = \Hom_k(\mcO(G),k) \]
is a cocommutative Hopf algebra in $\CV$. This sets up a covariant
equivalence of categories between finite group schemes in $\CV$ and
cocommutative Hopf algebras in $\CV$.

\section{Derivations}

This section expands on Section~\ref{se:derivations} in the context of $\CV$.

\begin{defn}
[Derivations in $\CVi$] 
Let $A$ be an algebra in $\CVi$ and $M$ an $A$-$A$-bimodule in $\CVi$.
A map $f\colon A \to M$ in $\CVi$ is a 
\emph{derivation} if 
\[ f(xy)=f(x)y+xf(y)+dx.(df)(y) \] 
where $df$ is defined by 
\[ (df)(x)=d(f(x))+f(dx). \] 
We write $\Der(A,M)$ for the $k$-vector space of 
derivations from $A$ to $M$, and $\Der(A)$ for $\Der(A,A)$.
\end{defn}

\begin{lemma}\label{le:dfxy}
If $f\colon A\to M$ is a derivation in $\CVi$ 
then for $x$,
$y\in A$ we have
\begin{enumerate} 
\item $(df)(xy) = d(f(x)).y+x.d(f(y))+f(dx).y+x.f(dy)$
\item $d(f'\circ f)=df'\circ f + f' \circ df$
\item $d(df)=0$.\smallskip

If $f,f'\colon A \to A$ are derivations then for $x\in A$ we have
\item $(df'\circ df)(x)=d(f'(d(f(x)))+d(f'(f(dx)))+f'(d(f(dx))).$
\end{enumerate}
\end{lemma}
\begin{proof}
For (i), we have
\begin{align*}
(df)(xy) & = d(f(xy))+f(d(xy)) \\
&= d(f(x).y+x.f(y)+dx.(df)(y))+f(dx.y+x.dy) \\
&=d(f(x)).y+f(x).dy+dx.f(y)+x.d(f(y)) \\
&\qquad +f(dx).y+dx.f(y)+f(x).dy+x.f(dy)\\
&=d(f(x)).y+x.d(f(y))+f(dx).y+x.f(dy).
\end{align*}
Parts (ii)--(iv) are similar but easier computations.
\end{proof}

\begin{defn}
If $f,f'\colon A\to A$ are derivations of an algebra in $\CVi$
we define
\[ [f,f']=f\circ f' + f' \circ f + df' \circ df. \]
\end{defn}

\begin{lemma}\label{le:[f,f']}
If $f,f'\colon A\to A$ are derivations of an algebra in $\CVi$ then so is
$[f,f']$.
\end{lemma}
\begin{proof}
This follows from 
Theorem~\ref{th:derivations}, but one can also give
a long and stupid proof with elements using Lemma~\ref{le:dfxy}.
\end{proof}

\section{Lie algebras}

Translating Definition~\ref{def:Lie}, 
we find that the definition in $\CVi$ is as follows.

\begin{defn}\label{def:Lie-algebra}
A \emph{Lie algebra} in $\CVi$ consists of an object $(L,d)$ in 
$\CVi$ together with a bilinear map of vector spaces, called the \emph{Lie bracket}:
\[ [\ ,\ ] \colon L \otimes L \to L \]
satisfying the following axioms:\smallskip
\begin{enumerate}
\item $d[x,y] =[dx,y]+[x,dy]$, \smallskip
\item $[y,x]=[x,y]+[dx,dy]$, \smallskip
\item $[x,[y,z]]+[y,[z,x]]+[z,[x,y]] =
  [dx,[y,dz]]+[dy,[z,dx]]+[dz,[x,dy]]$. 
\item If $dx=0$ then $[x,x]=0$.
\end{enumerate}
\end{defn}

\begin{rk}
Conditions (i)--(iii) of this definition come from
Definition~\ref{def:Lie} of an operadic Lie algebra. Condition~(iv)
comes from the following lemma. See also Remark~\ref{rk:res-Lie-algebra}.
\end{rk}

\begin{lemma}\label{le:[dx,dx]}
We have $[dx,dx]=0$ and $[x,dy]=[dy,x]$.
\end{lemma}
\begin{proof}
This follows from (ii) by replacing $y$ by $x$, respectively $dy$.
\end{proof}

\begin{rk}
In the presence of the other axioms, (iii) is symmetric, in
the sense that it is equivalent to
\[ [[x,y],z]+[[y,z],x]+[[z,x],y]=[[dx,y],dz]
+[[dy,z],dx]+[[dz,x],dy]. \]
\end{rk}

\begin{lemma}
We have $d[x,x]=0$.
\end{lemma}
\begin{proof}
Replace $y$ by $x$ in (i) and use Lemma~\ref{le:[dx,dx]}.
\end{proof}

\begin{lemma}\label{le:[x,y]}
If $A$ is an algebra in $\CVi$ then defining a bracket $[\ ,\ ]$
on $A$ via 
\[ [x,y] = xy + yx + dy.dx \]
we give $A$ the structure of a Lie algebra in $\CVi$.
\end{lemma}
\begin{proof}
We verify the four axioms in Definition~\ref{def:Lie-algebra}:

(i) $d[x,y]=d(xy+yx+dy.dx)=dx.y+x.dy+dy.x+y.dx=[dx,y]+[x,dy]$.

(ii) $[x,y]+[dx,dy]=xy+yx+dy.dx+dx.dy+dy.dx=xy+yx+dx.dy=[y,x]$.

(iii) We expand $[x,[y,z]]+[dx,[y,dz]]$ to get
\begin{gather*} 
x(yz+zy+dz.dy)+(yz+zy+dz.dy)x +(dy.z+y.dz+dz.y+z.dy)dx \\
{}+dx(y.dz+dz.y) +(y.dz+dz.y)dx 
\end{gather*}
The terms $y.dz.dx$ and $dz.y.dx$ occur twice and cancel.
We need to check that rotating $x$, $y$ and $z$ and adding the three
resulting expressions, we get zero. By symmetry, we only need check
the terms beginning with $x$ and with $dx$. Rotating the remaining
ten terms this way gives
\begin{equation*}
xyz+xzy+x.dz.dy+xyz+xzy+dx.dz.y+dx.y.dz+x.dz.dy+dx.y.dz+dx.dz.y
\end{equation*}
which adds up to zero.

(iv) If $dx=0$ then $xx + xx + dx.dx = 0$.
\end{proof}

\begin{rk}
Lemma~\ref{le:[x,y]} is the construction of the object $\Lie(A)$
described in~\eqref{eq:functors}.
\end{rk}

\begin{eg}\label{eg:End}
If $V$ is an object in $\CVi$ then $\End(V)$ is an algebra
in $\CVi$. Using Lemma~\ref{le:[x,y]}, we make $\End(V)$
into a Lie algebra.
\end{eg}

\begin{eg}
If $A$ be an algebra in $\CVi$ then by Lemmas~\ref{le:[f,f']} 
and~\ref{le:[x,y]}, $\Der(A)$ is a Lie subalgebra of $\End(V)$ in $\CVi$.
\end{eg}

\begin{rk}
Note the asymmetry in the lemma: we don't have $dx.dy=dy.dx$. 
It is possible to set everything up in such a way that
$[x,y]=xy+yx+dx.dy$. The effect of this would be to work with left
invariant derivations instead of right invariant ones in what follows.
\end{rk}

\begin{lemma}\label{le:Ad}
In the presence of Axioms {\rm\ref{def:Lie-algebra}\,(i)} and {\rm (ii)}, 
axiom {\rm (iii)} is equivalent to 
\[ [[x,y],z]=[x,[y,z]]+[y,[x,z]]+[dy,[dx,z]]. \]
which may be interpreted as saying that $\ad([x,y])(z)=[\ad(x),\ad(y)](z)$.
\end{lemma}
\begin{proof}
This is a long but straightforward exercise, using the definitions.
\end{proof}

\begin{defn}
The \emph{adjoint representation} of a Lie algebra $L$ in $\CVi$ is
the map from $L$ to $\End(L)$ sending $x$ to $\ad(x)$. Regarding $\End(L)$
as a Lie algebra as in Example~\ref{eg:End}, Lemma~\ref{le:Ad} shows
that the adjoint
representation is a homomorphism of Lie algebras in $\CVi$.
\end{defn}

\begin{defn}
Let $G$ be a group scheme in $\CVi$. A map $f\colon\mcO(G)\to
\mcO(G)$ is said to be 
\emph{right invariant} if the following diagram commutes:
\[ \xymatrix{\mcO(G) \ar[r]^f\ar[d]_\Delta & \mcO(G) \ar[d]^\Delta \\
\mcO(G) \otimes \mcO(G) \ar[r]_{f\otimes\Id} & \mcO(G)\otimes \mcO(G).} \]

It is easy to check that if $f$ and $f'$ are 
right invariant then so are $df$ and $f\circ f'$.
\end{defn}

\begin{lemma}
If $f$ and $f'$ are right invariant derivations of $\mcO(G)$ then so
is 
\[ [f,f']=f\circ f' + f'\circ f + df' \circ df. \]
This gives the set of right invariant derivations of $\mcO(G)$ the
structure of a Lie algebra.
\end{lemma}
\begin{proof}
The fact that $[f,f']$ is a derivation is proved in
Lemma~\ref{le:[f,f']}. The fact that it is right invariant follows from
the fact that each of $f\circ f'$, $f'\circ f$ and $df'\circ df$ are
separately right invariant. The fact that the right invariant
derivations form a Lie algebra now follows from Lemma~\ref{le:[x,y]}.
\end{proof}

\begin{defn}
For a group scheme $G$ in $\CVi$ we define $\Lie(G)$ to be
the Lie algebra in $\CVi$ of right invariant derivations $\mcO(G)\to
\mcO(G)$.
\end{defn}

The following verifies Theorem~\ref{th:main} independently in the case
of $\CVi$.

\begin{prop}
Let $G$ be a group scheme in $\CVi$.
There is a one to one correspondence between right invariant
derivations $\mcO(G)\to \mcO(G)$, elements of the tangent space 
$\Tan_{G,e}=(\mfm/\mfm^2)^*$, and elements of
$\Hom_{\CVAlg}(\mcO(G),\mcE)_e$ in $\CVi$.
\end{prop}
\begin{proof}
Lemma~\ref{le:f:k[S]to-k-ep} shows that if
$f\in\Hom_{\CVAlg}(\mcO(G),\mcE)_e$ then $f$ takes the
form
\[ f(a)= \eta(a)+f'(da)\ep + f'(a)d\ep \]
where $f'\colon \mcO(G)\to k$ is a linear map which takes value zero on
$1$ and on $\mfm_1^2$, and is an $\eta$-derivation in the sense
that it satisfies $f'(ab)=f'(a)\eta(b)+\eta(a)f'(b)$. In particular,
$f'$ can be regarded as a linear map $\mfm/\mfm^2\to k$.
Conversely, given $f'$, the map $f$ defined by $f'$ is an
element of $\Hom_{\CVAlg}(\mcO(G),\mcE)_e$.

Given an $\eta$-derivation $f'\colon \mcO(G)\to k$, we define a
map $F\colon \mcO(G)\to \mcO(G)$ via $F=(f'\otimes 1)\circ\Delta$
\[ \xymatrix{\mcO(G) \ar[r]^(.4)\Delta \ar[d]_F & \mcO(G) \otimes \mcO(G) 
\ar[d]^{f'\otimes\Id} \\ \mcO(G)\ar[r]^(.4)\cong&k\otimes\mcO(G).} \]
First we check right invariance:
\[ \xymatrix{\mcO(G) \ar[r]^(.4)\Delta\ar[d]_\Delta & 
\mcO(G) \otimes \mcO(G) \ar[d]^{\Delta\otimes\Id} \\
\mcO(G) \otimes \mcO(G) \ar[r]^(.4){\Id\otimes\Delta} \ar[d]_{f'\otimes \Id} & 
\mcO(G) \otimes \mcO(G)\otimes \mcO(G)\ar[d]^{f'\otimes \Id\otimes \Id} \\
k\otimes \mcO(G) \ar[r]^(.4){\Id \otimes \Delta} & k \otimes \mcO(G)\otimes \mcO(G)} \]
The upper square commutes by coassociativity of $\Delta$,
and the lower square obviously commutes. Therefore the outer rectangle
commutes, which is the definition of right invariance.

Next, we check that $F\colon \mcO(G)\to \mcO(G)$ is a derivation. Using
Sweedler's notation, we have
\begin{align*}
F(x)&=(f'\otimes \Id)\Delta(x)=f'(x_1)x_2, \\
(dF)(x)&=F(dx)+d(F(x)) \\
&=f'(x_1)dx_2+f'(dx_1)x_2+f'(x_1)dx_2\\
&=f'(dx_1)x_2.
\end{align*}
So
\begin{align*}
F(xy) &=(f'\otimes \Id)\Delta(xy) \\
&=(f'\otimes \Id)\Delta(x)\Delta(y) \\
&=(f'\otimes \Id)(x_1\otimes x_2)(y_1\otimes y_2) \\
&=(f'\otimes \Id)(x_1y_1\otimes x_2y_2+x_1\,dy_1\otimes dx_2\,y_2) \\
&=f'(x_1y_1)x_2y_2+f'(x_1\,dy_1)dx_2\,y_2\\
&=\eta(x_1)f'(y_1)x_2y_2+f'(x_1)\eta(y_1)x_2y_2
+\eta(x_1)f'(dy_1)dx_2\,y_2 \\
&= xF(y)+F(x)y+dx\,(dF)(y).
\end{align*}
This proves that $F$ is a derivation.

Conversely, given a right invariant derivation $F\colon \mcO(G)\to \mcO(G)$,
we define $f'\colon \mcO(G) \to k$ to be the composite $\eta\circ F$.
We check that $f'$ satisfies the properties above, and that 
these are inverse processes.
\end{proof}

\section{Restricted Lie algebras}

This section makes explicit the definitions of
Section~\ref{se:restricted} in the context of $\CVi$.

\begin{defn}\label{def:res-Lie-algebra}
A \emph{restricted Lie algebra} over $\CVi$ is defined
to be  a Lie algebra $L$ over $\CVi$ together with an operation
$x\mapsto x^{[2]}$ defined only on the kernel of $d$, and satisfying
the following identities:
\begin{enumerate}
\item[(i)] If $dx=dy=0$ then $(x+y)^{[2]}=x^{[2]}+y^{[2]}+[x,y]$,
\item[(ii)] If $\lambda\in k$ and $dx=0$, $(\lambda x)^{[2]}=\lambda^2x^{[2]}$
\item[(iii)] If $dx=0$ then $d(x^{[2]})=0$,
\item[(iv)] $[x,x]=(dx)^{[2]}$,
\item[(v)] If $dx=0$ then $[x,[x,y]]=[x^{[2]},y]$.
\end{enumerate}
\end{defn}

\begin{rk}\label{rk:res-Lie-algebra}
This is the translation of Definition~\ref{def:Fresse} into
$\CVi$. The operations $[x,y]$ and $x^{[2]}$ come from the fact that the
space $\Gamma^2L$ is spanned by the elements $x\otimes y + y\otimes x
+ dy\otimes dx$ and the elements $x\otimes x$ with $dx=0$.

Note that condition (iv) of Definition~\ref{def:Lie-algebra} follows
from condition (iv) of Definition~\ref{def:res-Lie-algebra}. This 
should be the case, since by Proposition~\ref{pr:forget}, the
functor of~\eqref{eq:functors} from restricted Lie algebras to operadic Lie
algebras factors through Lie algebras.
\end{rk}

\begin{theorem}\label{th:Ver-reslie}
If $G$ is a group scheme over $\CVi$ then 
$\Lie(G)$ is a restricted Lie algebra over $\CVi$
with $[x,y]=x\circ y+y\circ x+dy\circ dx$ 
and if $dx=0$ then $x^{[2]}=x \circ x$.
\end{theorem}
\begin{proof}
Let $x$ and $y$ be right invariant derivations of $\mcO(G)$.
\newline
(i) If $dx=dy=0$ then $(x+y)\circ(x+y)=x\circ x+y\circ y+[x,y]$.
\newline
(ii) If $dx=0$ then $(\lambda x)^{[2]}=(\lambda x)\circ(\lambda x)=\lambda^2(x\circ x)$.
\newline
(iii) If $dx=0$ then $d(x^{[2]})=d(x\circ x)= dx\circ x + x \circ dx = 0$.
\newline
(iv) $[x,x] = x\circ x + x\circ x + dx\circ dx = dx\circ dx = (dx)^{[2]}$.
\newline
(v) If $dx=0$ then 
\begin{align*}
[x,[x,y]] &= x\circ(x\circ y+y\circ x+dy\circ dx)+(x\circ y+y\circ
x+dy\circ dx)\circ x \\
&\qquad + (dx\circ y+x\circ dy+dy\circ x+y\circ dx)\circ dx \\
&= x\circ x\circ y+y\circ x\circ x =x^{[2]}\circ y + y\circ x^{[2]}= [x^{[2]},y].
\end{align*}
Note that if $dx\ne 0$ then bracketing with $x\circ x$ is not necessarily a
derivation.
\end{proof}

\begin{rk}
Exactly the same computations as in Theorem~\ref{th:Ver-reslie} show
that if $A$ is an associative algebra in 
$\CVi$ then we may define a restricted Lie algebra structure on $A$ by
setting $[x,y]=x.y+y.x+dy.dx$, and if $dx=0$ then $x^{[2]}=x.x$. This
is the dashed functor in Diagram~\eqref{eq:functors}.
\end{rk}

\begin{eg}
Let $\widetilde\bbG_a$ be the additive group scheme in $\CVi$, 
see Example~\ref{eg:Ga}. Then $\Lie(\bbG_a)$
has two basis elements, $e$ dual to $x$ and $f$ dual to $dx$. We have
$df=e$, $de=0$, all brackets are zero, and $e^{[2]}=0$. Since $df\ne
0$, $f^{[2]}$ is not defined. 
\end{eg}

\begin{eg} 
Let $\widetilde\bbG_m$ be the multiplicative group scheme in $\CVi$, see
Example~\ref{eg:Gm}. Then 
$\Lie(\bbG_m)$ has two basis elements, $e$ dual to $x-1$ and $f$ dual to
$dx$. We have $df=e$, $de=0$, all brackets are zero, 
$e^{[2]}=e$, and $f^{[2]}$ is not defined.
\end{eg}

\begin{eg}
Let $\GL(m+n|n)$ be the general linear group scheme in $\CVi$, see
Example~\ref{eg:GL}. Then $\Lie(\GL(m+n|n))$ has a basis consisting of
elements $e_{i,j}$ for $1\le i,j\le m+n$ dual to
$x_{i,j}-\delta_{i,j}$, and elements $f_{i,j}$ for $1\le i,j\le n$
dual to $dx_{i,j}$. We have $df_{i,j}=e_{i,j}$ if $1\le i,j\le n$,
and $de_{i,j}=0$. The brackets are given by
\begin{align*}
[e_{i,j},e_{k,\ell}]&=\delta_{j,k}e_{i,\ell}+\delta_{i,\ell}e_{j,k}\\
[e_{i,j},f_{k,\ell}]&=\delta_{j,k}f_{i,\ell}+\delta_{i,\ell}f_{k,j}\\
[f_{i,j},f_{k,\ell}]&=\begin{cases}
e_{i,j}& \text{if } (i,j)=(k,\ell) \\
\delta_{j,k}e_{i,\ell}+\delta_{i,\ell}e_{k,j}&\text{otherwise.}
\end{cases}
\end{align*}
The square map is given by $e_{i,j}^{[2]}=\delta_{i,j}e_{i,j}$, while
$f_{i,j}^{[2]}$ is not defined.
\end{eg}

\begin{rk}
Let $f\colon k \to k$ be the Frobenius map, given by $f(x) = x^2$.
For an object $A$ in $\CVAlg(k)$ the $r$th
\emph{Frobenius twist} $A^{(r)}$ is defined as the base change 
of $A$ over the Frobenius map: $A^{(r)} = k \otimes_{f^r} A$. 
Note that the $k$-linear map $A^{(r)} \to A$ given by $x\mapsto x^{2^r}$ 
is an algebra homomorphism for $r \ge 2$ but not for
$r=1$, because $(x+y)^2$ is not always equal to $x^2+y^2$. 
See the discussion of the Frobenius twist and the Frobenius kernels in Coulembier and
Sherman~\cite{Coulembier/Sherman:hstc}, and see also Section~6 of
Venkatesh~\cite{Venkatesh:2016a}.  
\end{rk}

\newcommand{\noopsort}[1]{}
\providecommand{\bysame}{\leavevmode\hbox to3em{\hrulefill}\thinspace}
\providecommand{\MR}{\relax\ifhmode\unskip\space\fi MR }
\providecommand{\MRhref}[2]{%
  \href{http://www.ams.org/mathscinet-getitem?mr=#1}{#2}
}
\providecommand{\href}[2]{#2}

\end{document}